\documentclass[a4paper,11pt,reqno]{amsart}
\usepackage{amssymb,amsmath,amsthm,amsfonts,mathtools,lmodern,mathrsfs,thmtools}

\usepackage[T1]{fontenc}
\usepackage[utf8]{inputenc}
\usepackage[english]{babel}
\usepackage[top=3cm,bottom=2.5cm,left=2.5cm,right=2.5cm]{geometry}
\usepackage{graphicx}
\usepackage{braket}
\usepackage[autostyle]{csquotes}
\usepackage[hypertexnames=false]{hyperref}
\hypersetup{
    hidelinks,
    colorlinks,
    linkcolor={red!50!black},
    citecolor={blue!50!black},
    urlcolor={blue!80!black},
}
\usepackage{enumerate}
\usepackage{xcolor}
\usepackage{thmtools}   
\usepackage[noabbrev,capitalize,nameinlink]{cleveref}

\usepackage{tikz-cd}

\theoremstyle{plain}
\newtheorem{theorem}{Theorem}[section]
\newtheorem{lemma}[theorem]{Lemma}
\newtheorem{corollary}[theorem]{Corollary}

\newtheorem{proposition}[theorem]{Proposition}

\theoremstyle{definition}
\newtheorem{definition}[theorem]{Definition}
\newtheorem{remark}[theorem]{Remark}

\newcommand{\R}{\mathbb R}
\newcommand{\Ran}{\mathbb R_{\mathrm{an}}}
\newcommand{\Ranexp}{\mathbb R_{\mathrm{an,exp}}}
\newcommand{\N}{\mathbb N}

\newcommand{\Z}{\mathbb Z}

\newcommand{\g}{\mathfrak{g}}

\newcommand{\spn}{\mathrm{span}}
\newcommand{\id}{e}
\newcommand{\de}{\mathrm{d}}
\newcommand{\End}{\mathrm{End}}
\newcommand{\Ad}{\mathrm{Ad}}
\newcommand{\ad}{\mathrm{ad}}
\newcommand{\abn}{\mathrm{Abn}}
\newcommand{\Lie}{\mathrm{Lie}}

\newcommand{\ddt}{\frac{\de}{\de t}}

\definecolor{deepgreen}{rgb}{0,0.6,0}

\author[E. Le Donne, A. Lerario, L. Nalon, N. Paddeu, L. Rizzi]{Enrico Le Donne, Antonio Lerario, Luca Nalon, Nicola Paddeu, Luca Rizzi}

 \address[Le Donne]{D\'epartement de Math\'ematiques, Ch. du mus\'ee 23, 1700 Fribourg (CH)}
 \email[Le Donne]{enrico.ledonne@unifr.ch}
 
 \address[Lerario]{SISSA, Via Bonomea 265, 34136 Trieste (IT)}
 \email[Lerario]{lerario@sissa.it}

 \address[Nalon]{SISSA, Via Bonomea 265, 34136 Trieste (IT)}
 \email[Nalon]{lnalon@sissa.it}

 \address[Paddeu]{D\'epartement de Math\'ematiques, Ch. du mus\'ee 23, 1700 Fribourg (CH)}
 \email[Paddeu]{nicola.paddeu@unifr.ch}

 \address[Rizzi]{SISSA, Via Bonomea 265, 34136 Trieste (IT)}
 \email[Rizzi]{lrizzi@sissa.it}

\thanks{E.~Le Donne, L.~Nalon, and N.~Paddeu were partially supported by the Swiss National Science Foundation 	(grant 200021-204501 `\emph{Regularity of sub-Riemannian geodesics and applications}'). L.~Nalon was partially supported by the Swiss National Science Foundation Postdoc.~Mobility Fellowship	(project number P500-2235462 `\emph{Lie groups of polynomial growth}'). L.~Rizzi was supported by the European Research Council (ERC) under the European Union’s Horizon 2020 research and innovation programme (grant agreement GEOSUB, No. 945655), and by the IRP project GEOSUBMAN by INSMI-CNRS. The authors also acknowledge the INdAM support.}

\begin{document}
	
	\title[Sard property in rank $2$ metabelian groups]{Sard property for rank 2 polarizations \\ in metabelian Lie groups}
	
	\begin{abstract} We provide sharp bounds on the dimension of the abnormal set for rank $2$ polarizations on metabelian Lie groups, establishing the Sard property for the end-point map of such groups. The proof is based on a novel approach that makes essential use of tools from tame geometry. We also obtain bounds for the dimension of the Goh-abnormal set for metabelian Lie groups where the codimension of the derived subgroup is at most $2$, with no assumption on the rank of the polarization. We thus infer that these polarized groups, equipped with sub-Riemannian structures, satisfy the minimizing Sard property.
	\end{abstract}
	\maketitle
	\tableofcontents
    
\section{Introduction}
Lie groups equipped with left-invariant geodesic distance functions are a classic object of study in both control theory and geometric group theory. By a result of Berestovskii \cite{MR995016}, all rectifiable curves in such a metric Lie group $G$ are tangent to a common sub-bundle 
\begin{equation*}
    \Delta = \bigcup_{g \in G} \left(\de L_g \right)_\id V
\end{equation*}
of the tangent bundle, for some bracket-generating subspace $V$ of the Lie algebra $\g = T_\id G$. Following the terminology of Gromov \cite[Section~0.1]{MR1421823}, we shall call the pair $(G,V)$ a \emph{polarized group} and we define its end-point map as
\begin{equation}  
    \End \colon L^1([0,1],V) \to G, \quad u \mapsto \gamma_u(1),
\end{equation}
where $\gamma_u$ is the absolutely continuous curve starting at the identity element $e \in G$ with derivative $\dot{\gamma}_u(t)=(\de L_{\gamma_u(t)})_\id u(t)$. If $u$ is a critical point of the end-point map we say that $\gamma_u$ is an \emph{abnormal curve}. The set of critical values of the end-point map is called the \emph{abnormal set} and it is denoted by $\abn(G,V)$. If it has zero measure, we say that $(G,V)$ satisfies the \emph{Sard property}.  Establishing the Sard property in polarized groups, and more generally in polarized manifolds, is a central open problem of sub-Riemannian geometry, \cite[Sec.~10.2]{MR1867362}, \cite[Prob.~ III]{A-openproblems}, \cite[Conj.~2]{RT-MorseSard}, with implications on the measure-theoretical properties of sub-Riemannian manifolds and geometric group theory, see e.g.\ \cite{MR3153949,MR4245620,MR5038044,borza2025failuremeasurecontractionproperty}. Despite the vast effort in the investigation of the Sard property, only partial results are known, see \cite{ZZ-Rigid,MR3352244, MR3569245,MR3981990,MR4104464,MR4524416,lerario2026sardpropertiespolynomialmaps,R-subdiff,BR-DukeMartinet,BFPR-StrongSardInventiones,BPR-Sardpreprint1,BPR-Sardpreprint2}.

Our main contribution is a sharp rectifiability result for the abnormal set in a class of metabelian (i.e., 2-step solvable) Lie groups, see \cref{def:metabelian}. Throughout the paper, $m$-rectifiability is defined  as in \cite[Definition~15.3]{MR1333890}, and it is understood in local charts. For convenience, we consider the empty set to be $m$-rectifiable, for every $m \in \Z$.
\begin{theorem} \label[theorem]{theo1_intro}
Let $(G,V)$ be a polarized Lie group such that $G$ is metabelian and $\dim(V)=2$. Then the abnormal set $\abn(G,V)$ is  $(\dim(G)-3)$-rectifiable.
In particular, $(G,V)$ satisfies the Sard property.
\end{theorem}

The latter result is noteworthy even when restricted to the case of metabelian stratified groups of rank two, i.e., polarized groups $(G,V)$ where $V$ is a two-dimensional first layer of a stratification, see \cite[Section~9.2]{donne2024metricliegroupscarnotcaratheodory}. The Lie algebras of the groups covered by \cref{theo1_intro} include all quotients of the free-metabelian Lie algebra $\mathfrak{m}$ on two generators. For this algebra, the dimensions of the layers of the lower central series $\{\mathfrak{m}_k\}_{k \in \mathbb{N}}$ grow linearly, specifically $\dim(\mathfrak{m}_k/\mathfrak{m}_{k+1})=k$, see \cite[Lemma~3.1]{MR2351997}. Consequently, \cref{theo1_intro} provides the first proof of the Sard property for a class of stratified groups where the nilpotency step and the dimensions of the layers in their stratifications is simultaneously unbounded, cf. \cite{MR4104464,MR4524416}. The estimate of \cref{theo1_intro} is sharp as a consequence of the following.

\begin{proposition}\label{prop:free-metabelian-sharpness}
Let $(M_{2,s},V)$ be the free-metabelian stratified group of rank $2$ and step
$s\geq 2$. Then
$\abn(M_{2,s},V)$ contains a smooth embedded submanifold of
dimension $\dim(M_{2,s})-3$. Consequently, $\abn(M_{2,s},V)$ is not $(\dim(M_{2,s})-4)$-rectifiable.
\end{proposition}

In sub-Riemannian geometry, the presence of energy-minimizing abnormal curves affects the regularity of sub-Riemannian distances, see \cite[Chapter~11]{MR3971262}. This fact motivates the study of the minimizing-abnormal set, which in a sub-Riemannian group $(G,d)$ is defined as:
\begin{equation}
\abn^{min}(G,d) \coloneq q \Set{ \gamma_u(1) : \text{$\de \End_u$ is singular, $\gamma_u$ energy-minimizer}}.
\end{equation}
The most general result so far is that such set is closed and nowhere dense \cite{agrasmoothness,RT-MorseSard}. If the minimizing-abnormal set has zero measure, the group $(G,d)$ is said to satisfy the \emph{minimizing Sard property}. In view of the Pontryagin Maximum Principle, the minimizing-abnormal set of a sub-Riemannian group is partitioned into the normal-abnormal set $\abn^{nor}(G,d)$ and the minimizing-strictly-abnormal set $\abn^{min,str}(G,d)$, we discuss these details in \cref{min_abn_prel}. An application of the Morse-Sard theorem to the sub-Riemannian exponential map proves that the normal-abnormal set $\abn^{nor}(G,d)$ has zero measure. Consequently, the challenge of establishing the minimizing Sard property rests  on the study of strictly-abnormal minimizers. 

Verifying whether a strictly-abnormal curve minimizes energy is generally a difficult task, particularly for non-regular examples, see \cite{CJMRSSS25, socio-alec-tommaso}. However, strictly-abnormal minimizers are necessarily Goh-abnormal, see \cite{AS-Morse} and \cite[Chapter 12]{MR3971262}. Such higher-order condition depends on the underlying polarization $(G,V)$ and not on the metric, see \cref{def:goh}. The second result of the current work provides a dimension bound for the Goh-abnormal set $\abn^{Goh}(G,V)$, defined as in \eqref{eq:goh_abnormal_set}, in a broader class of metabelian Lie groups.

\begin{theorem} \label[theorem]{thm:rank_two_intro}
    Let $(G,V)$ be a polarized Lie group such that $G$ is metabelian and such that $\dim(G/[G,G]) = 2$.  Then the Goh-abnormal set $\abn^{Goh}(G,V)$  is $(\dim(G)-3)$-rectifiable. 
    In particular, for every left-invariant, sub-Riemannian metric $d$ on $G$, the sub-Riemannian group $(G,d)$ satisfies the minimizing Sard property.
\end{theorem}

Standard arguments used to establish the negligibility of the normal-abnormal set actually demonstrate that the larger set of conjugate points is a subanalytic set of codimension one. Thus, the improved bound on the dimension of the Goh-abnormal set of \cref{thm:rank_two_intro} does not extend \emph{a priori} to the entire minimizing-abnormal set. Furthermore, general results regarding fine estimates for the dimension of the normal-abnormal set $\abn^{nor}(G,d)$ are currently unavailable. Consequently, direct bounds on the dimension of the whole abnormal set, such as the estimate provided in \cref{theo1_intro}, are useful to determine dimensional bounds of $\abn^{min}(G,d)$ as well.

The technical core of the present work lies in the machinery that we introduce to prove \cref{theo1_intro}. From an idea originated in \cite{MR3077915}, in all polarized groups, there exists a finite-dimensional family $\{P^\lambda\}_{\lambda \in \Lambda}$ of analytic functions such that each abnormal curve is contained in the zero-level set $Z^\lambda$ of some $P^\lambda$. We call these functions $\{P^\lambda\}_{\lambda \in \Lambda}$  the \emph{abnormal functions}, see \eqref{eq:def_plambda}. In metabelian Lie groups, some of these functions are invariant under the action of the derived subgroup. We use this observation to show that, in the groups covered by \cref{theo1_intro}, the set $H^\lambda$ of points reachable by rectifiable curves within $Z^\lambda$ is a one-dimensional object. In order to formalize the latter argument, we rely on the notion of tree-like equivalence for curves of bounded variations introduced in \cite{MR2630037}. 

Finally, in order to carry out a dimensional argument, we need a sufficiently tame description of how the sets $H^\lambda$ vary with the parameter $\lambda$. To achieve this, we rely on tools from tame geometry (see~\cite{MR1633348}) to obtain a subanalytic description of the family $Z^\lambda$. As a first (and expected) step, this reduces the problem to working within the o-minimal structure $\mathbb{R}_{\mathrm{an}}$ of globally subanalytic sets; see \cref{sec:tame-sets} and in particular \cref{rem:ran} for the distinction between analytic and globally subanalytic sets.
 The crucial additional point is that the corresponding end-point map, written in mixed coordinates, involves parameterized integration, see \cref{def_integral}. This operation is not handled directly inside $\Ran$ and forces us to pass to the larger o-minimal structure $\Ranexp$, where stability under integration is available. In this way, we obtain a countable family of $\Ranexp$-definable maps whose images cover the abnormal set, and whose domains have dimension at most $\dim(G)-3$. This is the key mechanism underlying the proof of \cref{theo1_intro}.

  \begin{remark}
     Note that, in the proof of both \cref{{theo1_intro}} and \cref{{thm:rank_two_intro}}, we actually prove the stronger fact that the abnormal set is contained in the countable union of definable sets of dimension at most $\dim(G)-3$  (in the o-minimal structure $\Ranexp$). In particular, the abnormal set is contained in the countable union of $C^1$-submanifolds of dimension at most $\dim(G)-3$.
 \end{remark}

 \begin{remark}
   We emphasize that \cref{theo1_intro} does not rule out the possibility that the abnormal set is dense. Nevertheless, as noted above, the minimizing abnormal set is known to be closed and nowhere dense.
\end{remark}

\medskip

\subsubsection*{\bf Structure of the paper} In \cref{prel}, we introduce the terminology and basic facts about abnormal sets and abnormal curves in polarized groups. In \cref{sec:tree}, we summarize some notions about tree-like curves. In \cref{sec:tame-sets}, we recall some facts of tame geometry and o-minimal structures that we combine  to obtain a regularity result for the end-point map in metabelian Lie groups, see \cref{def_integral}.

\cref{main} is devoted to the proofs of our main results: \cref{theo1_intro}, \cref{thm:rank_two_intro}. For the sake of completeness, we start by addressing the case of metabelian Lie groups such that $\dim(G/[G,G])=1$, for which most of the aforementioned machinery is not necessary and a stronger result is obtained, see \cref{thm:rank_one}. We end the paper with the proof of \cref{prop:free-metabelian-sharpness}.

\subsubsection*{\bf Use of artificial intelligence}

In accordance with the Leiden Declaration on Artificial Intelligence and Mathematics, endorsed by the International Mathematical Union, we disclose that OpenAI's ChatGPT (GPT-5.6 Pro) assisted in devising the family of abnormal curves used in the proof of \cref{prop:free-metabelian-sharpness}. The authors independently verified the argument, made all final mathematical decisions, and take full responsibility for the correctness and content of the manuscript.
	
\section{Preliminaries} \label{prel}

We recall some basic concepts on the sub-Riemannian geometry of polarized Lie groups, following \cite{donne2024metricliegroupscarnotcaratheodory}. Other classical references are \cite{MR3971262,MR1867362,MR3308395}.

\subsection{Polarized groups} Let $G$ be a 
Lie group, with identity element $\id$ and associated Lie algebra $\g \coloneq  T_\id G$. For $g \in G$, we denote by $L_g$ and $R_g$ the left and right translation by $g$, respectively. A subbundle $\Delta$ of the tangent bundle $TG$ is \emph{left-invariant} whenever
	\begin{equation*}
		\Delta_{g h} = \de L_g(\Delta_h) \quad \text{for every $g,h \in G$.}
	\end{equation*}
Fixed some $\Delta \subseteq TG$, we say that an absolutely continuous curve $\gamma \colon [0,1] \to G$ is \emph{horizontal} if $\dot\gamma(t)\in \Delta_{\gamma(t)}$ for a.e. $t\in[0,1]$.    
\begin{equation} \label{eq:def_distribution}
    \Delta_g \coloneq q \Set{ v \in T_gG : \left(\de L_g\right)^{-1}v \in V}, \quad \text{for every $g \in G$.}
\end{equation}
We refer to $\Delta$, defined as in \eqref{eq:def_distribution}, as the \emph{distribution} induced by $V$. 
    
\begin{definition}[Polarized group] Given a Lie group $G$, with Lie algebra $\g$, and vector subspace $V$, we say that the pair $(G,V)$ is a \emph{polarized group}, and we refer to $V$ as its \emph{polarization}. We also refer to $\Delta$, defined in \eqref{eq:def_distribution}, as the \emph{induced distribution} of $(G,V)$.
\end{definition}

We say that a polarization $V \subseteq \g$ is \emph{bracket-generating} if $\Lie(V)=\g$, where $\Lie(V)$ is the Lie algebra generated by $V$. If not stated otherwise, every polarization in this work is assumed to be bracket-generating. 

Next, we recall the correspondence between horizontal curves and their controls.
For each $u \in L^1([0,1],\g)$, we denote by $\gamma_u$ the absolutely continuous curve $\gamma \colon [0,1] \to G$ such that 
\begin{equation} \label{system}
 \begin{cases} 
 \gamma(0)=\id, \\ \dot{\gamma}(t)= \de L_{\gamma(t)}(u(t)), \quad \text{for a.e. $t \in [0,1]$.}
 \end{cases}
\end{equation}
Vice-versa, for every absolutely continuous curve $\gamma \colon [0,1] \to G$ such that $\gamma(0)=\id$, there exists a unique $u \in L^1([0,1],\g)$ such that \eqref{system} holds. We call $u$ the {\em control} of $\gamma$. If $\Delta$ is induced by a polarization $V$, the curve $\gamma_u$ is horizontal if and only if $u(t) \in V$ for a.e. $t \in [0,1]$.

We define the \emph{end-point map} of a polarized group $(G,V)$ as
\begin{equation} \label{eq:end_point}
    \End \colon L^1([0,1],V) \to G, \quad u \mapsto \gamma_u(1).
\end{equation}
We recall that the end-point map is smooth and it is surjective if and only if $G$ is connected and $V$ is bracket-generating.

\subsection{Abnormal sets in polarized groups}

Curves associated with controls that are  critical points of the end-point map \eqref{eq:end_point} are called \emph{abnormal curves}. The \emph{abnormal set} of a polarized group $(G,V)$ is the set of singular values of its end-point map
\begin{equation} \label{eq:abnormal_set}
    \abn(G,V) \coloneq q \Set{ \gamma_u(1) : \text{$\de \End_u$ is singular}}.
\end{equation}
We say that a polarized Lie group satisfies the \emph{Sard property} if $\abn(G,V)$ has zero Haar measure. The differential of the end-point map of a polarized groups has an explicit form. We refer to \cite[Proposition~7.2.1]{donne2024metricliegroupscarnotcaratheodory} for a proof of the following result.

\begin{proposition}[Differential of the end-point map] Let $(G,V)$ be a polarized group, then the differential of its end-point map is given by
    \begin{equation} \label{eq:differential_endpoint}
        \de \End_u (v) = \left(\de R_{\gamma_{u}(1)}\right) \int^1_0 \Ad_{\gamma_u(t)} v(t) \, \de t,
    \end{equation}
    where $\Ad_g \colon \g \to \g$ is the \emph{adjoint map} defined as $\Ad_g \coloneq q \de L_g \circ \de R^{-1}_g$, for every $g \in G$.
\end{proposition}

In view of \eqref{eq:differential_endpoint}, a curve $\gamma_u$ is abnormal if and only if there exists $\lambda \in \g^* \setminus \{0\}$ such that
\begin{equation} \label{eq:lambda_abnormal}
    \lambda \left( \int^1_0 \Ad_{\gamma_u(t)} v(t) \, \de t \right) = 0, \quad \text{for every $v \in L^1([0,1],V)$.}
\end{equation}
We say that a curve $\gamma_u$ is \emph{$\lambda$-abnormal} if \eqref{eq:lambda_abnormal} holds. We point out that being $\lambda$-abnormal is invariant under multiplication of $\lambda$ by a non-zero scalar. Therefore, it is sufficient to consider covectors in the unit sphere $\mathbb{S}(\g^*)$ with respect to some Euclidean norm on $\g^*$.

A standard approximation argument with Dirac masses shows that \eqref{eq:lambda_abnormal}, is equivalent to 
\begin{equation}
\label{eq:abn-equivalence}
    \lambda \left( \Ad_{\gamma_u(t)} X \right) = 0, \quad \text{for every $t \in [0,1]$ and $X \in V$}.
\end{equation}
In particular $\lambda$-abnormal curves exist if and only if $\lambda \in V^\perp\setminus \{0\}$, where $V^\perp$ denotes the annihilator of the set $V$. 

Given $\lambda \in \g^*$ and $X \in \g$, define the real-analytic\footnote{Every Lie group admits a canonical real-analytic structure for which the group operations are analytic. In particular, the conjugation map $(g,h)\mapsto ghg^{-1}$ is analytic on $G\times G$, hence its differential with respect to $h$ at the identity, namely $g\mapsto \Ad_gX$, is analytic for every fixed $X\in\g$. Composing with the linear functional $\lambda\in\g^*\setminus\{0\}$, we obtain that $P^\lambda_X$ is analytic.}
function
\begin{equation} \label{eq:def_plambda}
    P^\lambda_X \colon G \to \R, \quad g \mapsto \lambda \left( \Ad_g X \right),
    \end{equation}
    that we call \emph{abnormal functions}. Moreover, we also define the analytic set
\begin{equation*}
    Z^\lambda \coloneq q \Set{ g \in G \colon \text{$P^\lambda_X(g)=0$ for every $X \in V$}}.
\end{equation*}
We deduce that a curve $\gamma_u$ is $\lambda$-abnormal if and only if $\gamma_u \subseteq Z^\lambda$. We refer to the set
\begin{equation} \label{eq:reachable_set}
    H^\lambda \coloneq  \Set{ \End(u) \colon \begin{gathered}
u\in L^1([0,1],V), \ 
\text{$\gamma_u(t) \in Z^\lambda$ for every $t \in [0,1]$}  \end{gathered}} ,
\end{equation}
for $\lambda \in V^\perp \setminus \{0\}$ as the \emph{reachable set} of $Z^\lambda$. Both $Z^\lambda$ and $H^\lambda$ are invariant under rescaling of $\lambda$ by a non-zero constant. We summarize the previous discussion in the following statement.

\begin{proposition}[Description of the abnormal set in polarized groups] Let $(G,V)$ be a polarized group. The abnormal set  satisfies
\begin{equation*}
    \abn(G,V) = \bigcup_{\lambda \in \mathbb{S}(V^\perp)} H^\lambda 
\end{equation*} 
where $H^\lambda$ is defined in \eqref{eq:reachable_set}.   
\end{proposition}

We conclude this section with some properties of functions as in \eqref{eq:def_plambda}, as well as discussing how the abnormality condition behaves under quotients by normal Lie subgroups.

\begin{remark} \cite[Section~7.3.2]{donne2024metricliegroupscarnotcaratheodory}\label[remark]{rem:derivative_plambda}
     For every $\lambda \in \g^*$ and $X,Y \in \g$, we get $X P^\lambda_Y = P^\lambda_{[X,Y]}$. 
    Therefore, if $t \in [0,1]$ is a point of differentiability of a curve $\gamma_u$, with $u \in L^1([0,1],\g)$, we obtain
    \begin{equation} \label{eq:diff}
        \ddt P^\lambda_X(\gamma_u(t)) = P^\lambda_{[u(t),X]} (\gamma_u(t)).
    \end{equation}
\end{remark}

\begin{remark} \label[remark]{rem:abnormal_ideal}
    We claim that, for every $\lambda \in \g^*$, the set $I^\lambda \subseteq \g$, defined as
    \begin{equation*}
        I^\lambda \coloneq q \Set{ X \in \g : \text{$P^\lambda_X(g)=0$ for every $g \in G$}},
    \end{equation*}
    is an ideal of $\g$. Indeed, since $\Ad$ is linear, we get that $I^\lambda$ is a vector subspace. Moreover, if $Y \in I^\lambda$, then $P^\lambda_Y \equiv 0$ and from \cref{rem:derivative_plambda} we get $0 \equiv XP^\lambda_Y = P^\lambda_{[X,Y]}$, for every $X \in \g$.
\end{remark}

\begin{proposition} \label[proposition]{rem:quotient}
    Let $(G,V)$ be a polarized group with Lie algebra $\g$, and $N \le G$ be a normal Lie subgroup of $G$ with Lie algebra $I \subseteq \g$. Then $(G/N,V/I)$ is a polarized group. Denote by $\pi \colon G \to G/N$ the quotient map and let $\gamma_u$ be an horizontal curve. Then $\pi \circ \gamma_u$ is horizontal for $(G/N,V/I)$ and it is abnormal if and only if $\gamma_u$ is $\lambda$-abnormal for some $\lambda \in \mathbb{S}(I^\perp)$.
\end{proposition}

\begin{proof}
    Clearly, $G/N$ is a Lie group and $V/I= \de\pi(V)$ is bracket-generating whenever $V$ is, since each surjective Lie algebra morphism maps bracket-generating sets in bracket-generating sets. Moreover, if $\gamma_u$ be an horizontal curve in $(G,V)$, then
    \begin{equation*}
        \ddt (\pi \circ \gamma_u) = \de \pi \circ \de L_{\gamma_u(t)}(u(t))= \de L_{\pi(\gamma_u(t))} \de \pi(u(t)) \in  \de L_{\pi(\gamma_u(t))}(\de\pi(V))= \de L_{\pi(\gamma_u(t))}(V/I),
            \end{equation*}
        where in the second equality we used that $\pi$ is a Lie group morphism. For the second part of the statement, we use the canonical identification $\left(\de \pi(\g)\right)^* \cong \ker(\de \pi)^\perp = I^\perp$ and that, since $\pi$ is a Lie group morphism, we have $\Ad_{\pi(g)}(\de \pi(X)) = \de \pi \left(\Ad_g(X)\right)$ for every $g \in G$ and $X \in \g$. Thus
        \begin{equation} \label{poly_proj}
            P^\lambda_X = P^\lambda_{\de\pi(X)} \circ \pi, \quad \text{for every $\lambda \in \mathbb{S}(I^\perp)$.}
        \end{equation}
        Fix $\lambda \in \mathbb{S}(I^\perp)$. On the one hand $\pi \circ \gamma_u$ is $\lambda$-abnormal if and only if $P^\lambda_{\de \pi(X)}(\pi \circ \gamma_u(t))=0$ holds for every $t \in [0,1]$ and $X \in V$. On the other hand, $\gamma_u$ is $\lambda$-abnormal if and only if $P^\lambda_X(\gamma_u(t))=0$ holds for every $t \in [0,1]$ and $X \in V$. In view of \eqref{poly_proj}, the two conditions are equivalent.
\end{proof}

\subsection{Abnormal minimizers and Goh-abnormal curves} \label{min_abn_prel}

A scalar product $\langle\cdot,\cdot\rangle$ on the polarization $V$ of a connected polarized group $(G,V)$ endows $G$ with a left-invariant distance function $d$ such that
\begin{equation} \label{eq:distance}
    d(\id,g) \coloneq  \inf \Set{\left( \int^1_0 \langle u(t),u(t)\rangle \de t\right)^{1/2}\colon u \in L^1([0,1],V), \, \gamma_u(1)=g }.
\end{equation}
The resulting metric space $(G,d)$ is called \emph{sub-Riemannian group}, and the datum $(G,V,\langle\cdot,\cdot\rangle)$ is its associated \emph{sub-Riemannian structure}. The infimum in \eqref{eq:distance} is always realized and we refer to a curve $\gamma_u$, with $u \in L^1([0,1],V)$, realizing it as \emph{energy-minimizing curve}.

The importance of abnormal curves stems from the fact that they satisfy the first-order necessary condition for being energy-minimizer provided by the Pontryagin's Maximum Principle, which for sub-Riemannian Lie groups is a simple consequence of the Lagrange's Multipliers Rule and it states as follows.

\begin{theorem}[Pontryagin’s Maximum Principle {\cite[Theorem~7.3.3]{donne2024metricliegroupscarnotcaratheodory}}]
\label[theorem]{thm:pont}
Let $(G,d)$ be a sub-Riemannian group with sub-Riemannian structure $(G,V,\langle\cdot,\cdot\rangle)$. If $\gamma_u \colon [0,1] \to G$ is an energy-minimizing curve, then at least one of the following holds:
\begin{enumerate}[(i)]
    \item there exist $\lambda \in \g^*$ such that $\langle u(t),X\rangle = \lambda\left(\Ad_{\gamma_u(t)}X\right)$ for every $X \in V$ and $t \in [0,1]$, 
    \item $\gamma_u$ is abnormal.
\end{enumerate}
\end{theorem}

Curves satisfying (i) in \cref{thm:pont} are called \emph{normal trajectories} and they are known to be locally energy-minimizing, see \cite{MR1867362}. Nevertheless, there are energy-minimizing curves that are not normal trajectories and therefore they are abnormal curves, see \cite{MR1297101}. We refer to abnormal energy-minimizing curves that are not normal as \emph{strictly-abnormal minimizers}. Strictly-abnormal minimizers satisfy a second-order condition that depends solely on the polarization.

\begin{definition}[Goh-abnormal] \label[definition]{def:goh}
    Let $(G,V)$ be a polarized group. We say that a curve $\gamma_u$ is \emph{Goh-abnormal} if $u \in L^1([0,1],V)$ is a critical point for the end-point map of the polarized group $(G,V+[V,V])$. The \emph{Goh-abnormal set} is defined as follows:
    \begin{equation} \label{eq:goh_abnormal_set}
    \abn^{Goh}(G,V) \coloneq q \Set{ \gamma_u(1) : \text{$\gamma_u$ is a Goh-abnormal curve}}.
\end{equation}
\end{definition}

\begin{theorem}[Goh condition {\cite[Theorem~12.12]{MR3971262}}]
    Let $(G,d)$ be a sub-Riemannian group. If a curve $\gamma_u \colon [0,1] \to G$ is a strictly-abnormal minimizer, then $\gamma_u$ is Goh-abnormal.
\end{theorem}
 
We say that a sub-Riemannian group $(G,d)$ with structure $(G,V,\langle\cdot,\cdot\rangle)$ satisfies the \emph{minimizing Sard property} if the \emph{minimizing-abnormal set}, defined as
\begin{equation} \label{eq:min_abnormal_set}
    \abn^{min}(G,V) \coloneq q \Set{ \gamma_u(1) : \text{$\de \End_u$ is singular, $\gamma_u$ energy-minimizer}},
\end{equation}
has zero measure. A direct application of the Morse-Sard theorem in the finite-dimensional case ensures that the \emph{normal-abnormal set}, defined as 
\begin{equation} \label{eq:norm_abnormal_set}
    \abn^{nor}(G,d) \coloneq q \Set{ \gamma_u(1) : \text{$\de \End_u$ is singular, $\gamma_u$ normal trajectory}},
\end{equation}
is a subanalytic set of dimension at most $\dim(G)-1$. We thus obtain the following result.

\begin{proposition}
    Let $(G,V)$ be a polarized group and assume that the Goh-abnormal set $\abn^{Goh}(G,V)$ has zero measure. Then, for every scalar product $\langle\cdot,\cdot\rangle$ on $V$, the sub-Riemannian group $(G,d)$ with structure $(G,V,\langle\cdot,\cdot\rangle)$ satisfies the minimizing Sard property.
\end{proposition}

In the case of 2-dimensional polarizations, the Goh condition is automatically verified, as explained by the following result, which is a direct consequence of \eqref{eq:diff}.

\begin{proposition} \cite[Proposition~7.3.5]{donne2024metricliegroupscarnotcaratheodory}\label[proposition]{prop:goh_2}
    Let $(G,V)$ be a polarized group such that $\dim(V)=2$, then a non-constant horizontal curve $\gamma_u$ is abnormal if and only if it is Goh-abnormal.  
\end{proposition}

\subsection{Mixed coordinates in solvable Lie groups}\label{sec:mixed} A group $G$ is \emph{$s$-step solvable} if its derived series $(D^n(G))_{n \in \N}$, recursively defined as
\begin{equation*}
    D^0(G)\coloneq  G, \quad D^{n+1}(G)\coloneq  [D^n(G),D^n(G)] = \langle ghg^{-1}h^{-1}\colon g,h \in D^n(G)\rangle, \quad \text{for $n \in \N$,}
\end{equation*}
satisfies $D^s(G)=\{\id\}$. 

\begin{definition}[Metabelian group]\label[definition]{def:metabelian}
    A group $G$ is \emph{metabelian} if it is $2$-step solvable. Equivalently, $G$ is metabelian if its derived subgroup
   $
    G' \coloneq q  D^1(G)$
    is abelian.
\end{definition}

It is well known that a connected Lie group $G$ is $s$-step solvable if and only if its associated Lie algebra $\g$ is, that is, if and only if its derived series $(D^n(\g))_{n \in \N}$, recursively defined as
\begin{equation*}
    D^0(\g)\coloneq  \g, \quad D^{n+1}(\g) \coloneq  [D^n(\g),D^n(\g)]= \spn\Set{[X,Y]\colon X,Y \in D^n(\g)}, \quad \text{for $n \in \N$,}
\end{equation*}
satisfies $D^s(\g)=\{0\}$, see \cite[Theorem~11.2.5]{hilgert2011structure}. Similarly, a Lie algebra $\g$ is \emph{metabelian} if its derived subalgebra $\g' \coloneq  D^1(\g)$ is abelian. 

We recall that a subgroup $H \subseteq G$ is an \emph{integral subgroup} if it is of the form $H = \langle \exp(\mathfrak{h})\rangle$ for some Lie subalgebra $\mathfrak{h} \subseteq \g$, and that it is normal if and only if $\mathfrak{h}$ is an ideal. We make use of the following fact, for which refer to \cite[Proposition~11.2.4 and Proposition~11.2.15]{hilgert2011structure}. 

\begin{remark} \label[remark]{rem:integral}
    In a connected Lie group $G$, elements of the derived series $(D^n(G))_{n \in \N}$ are normal integral subgroups and $D^n(G) = \langle \exp(D^n(\g))\rangle$, for every $n \in \N$. Moreover, if $G$ is solvable and simply connected, then integral subgroups are closed and simply connected. In particular, they are Lie subgroups.
\end{remark}

Let $G$ be a simply connected, solvable Lie group and denote by $\g$ its Lie algebra. We point out that, in contrast with the nilpotent case, the exponential map $\exp \colon \g \to G$ may fail to be a global diffeomorphism. Nevertheless, if $B \subseteq \g$ is a vector subspace in direct sum with $\g'\coloneq  [\g,\g]$, then the map 
\begin{equation} \label{eq:mixed_coordiantes}
    \Phi \colon \g \cong \g' \times B \to G, \quad (a,b) \mapsto \exp(a)\exp(b)
\end{equation}
is a global diffeomorphism, see \cite[Lemma~14.3.6]{hilgert2011structure}. In this regard, we stress that $G$ is solvable, thus $G'\coloneq  \exp(\g')$ is a nilpotent Lie subgroup, see \cite[Corollary~5.4.12]{hilgert2011structure}. Being $G'$ also simply connected from \cref{rem:integral}, we get that the map $\exp \colon \g' \to G'$ is a global diffeomorphism. If $g = \Phi(a,b)$, we say that $(a,b)$ are the \emph{mixed coordinates} of $g$. 

\begin{remark} \label[remark]{poly_invariance}
   Let $G$ be a metabelian, simply connected Lie group with Lie algebra $\g$ and mixed coordinates as in \eqref{eq:mixed_coordiantes}.
   Fix $a,b\in \g'\times B$.
   We claim that
    \begin{equation*}
        P^\lambda_Y(a,b) = P^\lambda_Y(0,b), \quad \text{for every $\lambda \in \g^*$  and $Y \in \g'$.}
    \end{equation*}
    Indeed, we recall that $\Ad:G\to \mathrm{Aut}(\g)$ is a group homomorphism. Therefore, omitting the map $\Phi$ from the notation, we get $$P^\lambda_Y(a,b) = \lambda \left( \Ad_{(a,b)}Y\right) = \lambda \left( \Ad_{\exp(a)}\Ad_{\exp(b)} Y \right) = \lambda \left(\Ad_{\exp(b)} Y \right) = P^\lambda_Y(0,b),$$ 
    where we used that, since $Y \in \g'$, then $\Ad_{\exp(b)} Y \in \g'$, together with the fact that $G'$ is abelian; thus, its adjoint map is trivial.     
    \end{remark}

\subsection{End-point map in solvable Lie groups} 
Let $G$ be a simply connected, solvable Lie group with Lie algebra $\mathfrak{g}$. Fix a subspace $B \subseteq \mathfrak{g}$ such that $\mathfrak{g} = \mathfrak{g}' \oplus B$, and consider the mixed coordinates on $\mathfrak{g}' \times B$ as in \eqref{eq:mixed_coordiantes}. For each $(a,b) \in \mathfrak{g}' \times B$, let $L_{(a,b)} \colon \mathfrak{g} \to \mathfrak{g}$ denote the left-translation map expressed in mixed coordinates, specifically $L_{(a,b)} \coloneq  \Phi^{-1} \circ L_{\Phi(a,b)} \circ \Phi$. 
We denote its differential at the origin by $\de L_{(a,b)} \colon \mathfrak{g} \to \mathfrak{g}$.
We have that the mixed coordinates of a curve $\gamma_u$ associated with a control $u \in L^1([0,1], \mathfrak{g})$ are determined by the following ODE:
\begin{equation} \label{eq:ode}
\frac{\de}{\de t} \Phi^{-1}(\gamma_u(t)) = \de L_{\Phi^{-1}(\gamma_u(t))}(u(t)), \quad \Phi^{-1}(\gamma_u(0)) = 0.
\end{equation}
Equivalently, in integral form:
\begin{equation} 
\label{eq:integral_ode}\Phi^{-1}(\gamma_u(t)) = \int^t_0 \de L_{\Phi^{-1}(\gamma_u(\tau))}(u(\tau)) \, \de \tau.
\end{equation}
Notably, in the metabelian setting, equation \eqref{eq:ode} admits an explicit solution. In the following discussion, we provide the details in the case where $u \in L^1([0,1], B)$, which is sufficient for our purposes. First, we define the following entire functions

\begin{equation} \label{eq:power_series}
    \varphi(z) \coloneq q \frac{1 - e^{-z}}{z} = \sum_{k=0}^\infty \frac{(-1)^k}{(k+1)!}z^k,  \quad \rho(z) \coloneq q e^z - \varphi(-z) = \sum_{k=0}^\infty \frac{k}{(k+1)!}z^k.
\end{equation}

\begin{lemma}
    For each $(a,b) \in \mathfrak{g}' \times B$ and $X \in B$, it holds
    \begin{equation} \label{eq:differential}
        \de L_{(a,b)}(X)= ( \rho(\ad_b)X,X) \in \g' \times B,
    \end{equation}
    where for $Y,Z\in \g$, we denote $\ad_YZ\coloneq [Y,Z]$.
    In particular, if $(a(t),b(t))$ are the mixed coordinates of a curve $\gamma_u$ with $u \in L^1([0,1], B)$, then
    \begin{equation} \label{eq:curve_coordinates}
\begin{cases}
a(t) = \int^t_0 \rho(\ad_{b(\tau)}) u(\tau) \de \tau, \\
b(t) = \int_0^t u(\tau) \de \tau.
\end{cases}
\end{equation}
\end{lemma}

\begin{proof}
    First, we observe that $\exp(b)\exp(tX)\exp(-b-tX) \in G'$, for every $t \in [0,1]$, since it is a trivial element in the abelian quotient $G/G'$. Next, we compute
    \begin{align*}
        \ddt \bigg|_{t=0 } \exp(b)\exp(tX)\exp(-b-tX) &= \Ad_{\exp(b)}X - \de L_{\exp(b)} \de L^{-1}_{\exp(b)}\left(\varphi(-\ad_b)\right)X \\
        &= e^{\ad_b}X - \varphi(-\ad_b)X = \rho(\ad_b)X,
    \end{align*}
    where in the first equality we used the formula for the differential of the exponential function, see \cite[Proposition~3.4.2]{hilgert2011structure}. Since $G'$ is abelian and $a, \rho(\ad_b)X \in \g'$, we infer that
    \begin{align*}
        L_{\Phi(a,b)}\exp(tX) &= \exp(a)\exp(b)\exp(tX) \\ &= \exp(a+ t \rho(\ad_b)X + o(t))\exp(b+tX) \\
        &= \Phi(a+ t \rho(\ad_b)X + o(t),b+tX).
    \end{align*}
    By differentiating the latter equation at $t=0$ we obtain \eqref{eq:differential}. The second part of the statement follows from \eqref{eq:integral_ode} by using \eqref{eq:differential}.    
\end{proof}    

\section{Multiplicative integrals and tree-like curves} \label{sec:tree}
In our discussion, it is useful to interpret the end-point map in terms of the multiplicative integral. If what follows, every curve is assumed to be absolutely continuous.

\begin{definition}[Multiplicative integral]
\label{def:multiplicative_integral}
Let $G$ be a Lie group, with Lie algebra $\g$. The \emph{multiplicative integral} $\gamma \colon [0,1] \to G$ of a curve $\sigma \colon [0,1] \to \g$ is the curve satisfying \eqref{system}, with $u = \dot{\sigma}$. 
\end{definition}

Hereafter, given a curve $\sigma \colon [0,1] \to \g$, we denote by $\overline{\sigma}$ its \emph{inverse}, i.e. $\overline{\sigma}(t) \coloneq  \sigma(1-t)$. Given two curves $\sigma_1, \sigma_2 \colon [0,1] \to  \g$, we say that $\sigma_2$ is compatible with $\sigma_1$ if $\sigma_1(1)=\sigma_2(0)$, and in that case we define their \emph{concatenation} $\sigma_1 * \sigma_2$ as the curve
\begin{equation*}
    \sigma_1 * \sigma_2 \colon [0,1]\to \g \quad  t \mapsto \begin{cases}
        \sigma_1(2t)  & \text{if $t \le \tfrac{1}{2}$,} \\ \sigma_2(2t-1) &\text{if $t > \tfrac{1}{2}$.}
    \end{cases}
\end{equation*}

We summarize how multiplicative integrals behave under concatenation and inverse of curves.

\begin{remark} \label{rem:path}
 Let $\gamma \colon [0,1] \to G$ be the multiplicative integral of some curve $\sigma \colon [0,1] \to \g$. We claim that the multiplicative integral of the inverse curve $\overline{\sigma}$ of $\sigma$ is the curve 
\begin{equation*}
    \overline{\gamma} \colon [0,1] \to G, \quad t \mapsto \gamma(1)^{-1}\gamma(1-t).
\end{equation*}
Indeed, $\overline{\gamma}(0) = \id$ and for a.e. $t \in [0,1]$ we get
\[
\ddt \gamma(1)^{-1}\gamma(1-t) = \de L_{\gamma(1)^{-1}}\left(\ddt \gamma(1-t) \right) = \de L_{\gamma(1)^{-1}\gamma(1-t)} \left(-\ddt \sigma(1-t) \right) = \de L_{\overline{\gamma}(t)} \left(\ddt \overline{\sigma}(t)\right).
\]
\end{remark}

\begin{remark} \label{rem:path2}
Let $\gamma_1, \gamma_2 \colon [0,1] \to G$ be the multiplicative integrals of some compatible curves $\sigma_1, \sigma_2 \colon [0,1] \to \g$, respectively. We claim that the multiplicative integral $\gamma$ of the concatenation $\sigma_1 * \sigma_2$ satisfies 
\begin{equation*}
    \gamma\colon [0,1] \to G, \quad t \mapsto \begin{cases}
       \gamma_1(2t)  & \text{if $t \le \tfrac{1}{2}$,} \\ \gamma_1(1)\gamma_2(2t-1) &\text{if $t > \tfrac{1}{2}$.} \end{cases}
\end{equation*}
Indeed, $(\gamma)(0) = \id$ and for a.e. $t \in [0,1]$ we get
\begin{align*}
    \ddt \gamma_1(2t) &= \de L_{\gamma_1(2t)}(2\dot{\sigma}_1(2t))= \de L_{(\gamma)(t)}\left(\ddt (\sigma_1*\sigma_2)(t)\right), & &\text{if $t \le \tfrac{1}{2}$,} \\
        \ddt \gamma_1(1)\gamma_2(2t-1) &= \de L_{\gamma_1(1)} \de L_{\gamma_2(2t-1)}(2\dot{\sigma}_2(2t-1)) = \de L_{(\gamma)(t)}\left(\ddt (\sigma_1*\sigma_2)(t)\right), & &\text{if $t > \tfrac{1}{2}$.} 
\end{align*}
In particular, for the endpoint it holds $\gamma(1)=\gamma_1(1)\gamma_2(1)$.
\end{remark}

We next introduce the notion of tree-like equivalence between horizontal curves, which provides a sufficient criterion for two rectifiable curves to have the same end-point.

\begin{definition}[Tree-like curve and tree-equivalence]
    Let $(W,\lVert \cdot \rVert)$  be a normed vector space. A (closed) curve $\sigma \colon [0,1] \to W$ is \emph{tree-like} if there exists a continuous function $h \colon [0,1] \to [0,+\infty)$, with $h(0)=h(1)=0$, such that 
    \begin{equation*}
        \lVert \sigma(t) -\sigma(s) \rVert \le h(t) + h(s) - 2 \inf_{\tau \in [s,t]} h(\tau), \quad \text{for every $0 \le s \le t \le 1$.}
    \end{equation*}
    Two curves $\sigma_1, \sigma_2 \colon [0,1] \to W$, with $\sigma_1(1)=\sigma_2(1)$, are \emph{tree-equivalent} if $\sigma_1 * \overline{\sigma}_2$ is tree-like.
\end{definition}

Tree-like curves were introduced in \cite{MR2630037} to characterize paths of bounded variation having the same signature, generalizing a result by Chen \cite{MR106258}.
The deep characterization result by Hambly and Lyons has an equivalent formulation in terms of multiplicative integrals, see \cite[Corollary~1.7]{MR2630037} where the authors use the terminology ``development of a path'' in place of ``multiplicative integral''. In the present work, however, we only need one direction of this characterization, which is straightforward and it is obtained by combining \cite[Lemma~6.3 and Corollary~6.4]{MR2630037} with \cref{rem:path} and \cref{rem:path2}.

\begin{proposition} \label[proposition]{trivial_integral}
    Let $G$ be a Lie group with Lie algebra $\g$. If $\sigma \colon [0,1] \to \g$ is a tree-like curve and $\gamma$ is the multiplicative integral of $\sigma$, then $\gamma(1)=\id$. In particular, if $\sigma_1,\sigma_2 \colon [0,1] \to \g$ are tree-equivalent and $\gamma_1,\gamma_2$ are their respective multiplicative integrals, then $\gamma_1(1)=\gamma_2(1)$.
\end{proposition}

We recall the notion of geometric realization of a graph from \cite[Section~1.6]{MR3753580} and we extend it to subsets of $\R^n$.

\begin{definition} \label[definition]{geo_real}
Let $K$ be a finite 1-dimensional simplicial complex and $Z \subseteq \R^n$, with the subspace topology. A \emph{geometric realization of $Z$ as finite graph} is an homeomorphism $\varphi \colon | K| \to Z$, where $|K|$ denotes the polyhedron of the complex.
\end{definition}

If $Z \subseteq \R^n$ is connected and it admits a geometric realization as finite graph that is absolutely continuous on each 1-simplex, then the induced path metric $d$ on $Z$ is geodesic and compatible with the Euclidean topology. Moreover, the universal cover $\widehat{Z}$ of $Z$ admits a compatible geodesic distance function $\widehat{d}$ such that the covering map $\pi \colon \widehat{Z} \to Z$ is a submetry. The metric space $(\widehat{Z},\widehat{d})$ is a \emph{tree}. 

\begin{proposition} \label[proposition]{prop:homotopy}
    Let $G$ be a Lie group with Lie algebra $\g$, and $Z \subseteq \g$ be connected and admitting a geometric realization as finite graph  that is absolutely continuous on each 1-simplex. If a closed curve $\sigma \colon [0,1] \to Z$ has trivial homotopy (in $Z$), then $\sigma$ is tree-like. In particular, if $\sigma_1,\sigma_2 \colon [0,1] \to Z$, with $\sigma_1(0)=\sigma_2(0)$ and $\sigma_1(1)=\sigma_2(1)$ are homotopically equivalent (in $Z$), then $\gamma_1(1)=\gamma_2(1)$, where $\gamma_1$ and $\gamma_2$ are the multiplicative integrals of $\sigma_1$ and $\sigma_2$, respectively.
    \end{proposition}

\begin{proof}
    Fix an Euclidean norm $\lVert \cdot \rVert$ on $\g$ and consider on $Z$ the geodesic distance function $d$ induced by the inclusion $Z \xhookrightarrow{} \g$. We lift this metric to the universal cover $(\widehat{Z},\widehat{d})$, so that the covering map $\pi \colon \widehat{Z} \to Z$ is a submetry. Consider a closed curve $\sigma \colon [0,1] \to Z$ with trivial homotopy (in $Z$), then its lift $\hat{\sigma} \colon [0,1] \to \hat{Z}$ is also a closed curve. We define the function
    \begin{equation} \label{def_h}
        h \colon [0,1] \to [0,+\infty), \quad t \mapsto \hat{d}(\hat{\sigma}(0),\hat{\sigma}(t)).
    \end{equation}
We point out that $h$ is continuous and that $h(0)=h(1)=0$, since $\hat{\sigma}$ is a closed curve. 

Fix $0 \le s \le t \le 1$. Since $(\hat{Z},\hat{d})$ is a tree, the curve $\hat{\sigma}|_{[s,t]}$ passes through the center of the tripod $T(\hat{\sigma}(0),\hat{\sigma}(s),\hat{\sigma}(t))$, see \cite[Lemma~11.13]{MR3753580}. In particular, there is $\tau \in [s,t]$ such that 
\begin{equation*}
    \hat{d}(\hat{\sigma}(t),\hat{\sigma}(s)) = h(t) + h(s) - 2h(\tau).
\end{equation*}
We then estimate
\begin{equation*}
    \lVert \sigma(t) -\sigma(s) \rVert \le d(\sigma(t),\sigma(s)) \le \hat{d}(\hat{\sigma} (t),\hat{\sigma}(s)) \le h(t) + h(s) - 2 \inf_{\tau \in [s,t]} h(\tau).
\end{equation*}
    We thus proved that $\sigma$ is a tree-like curve in $\g$, with function $h$ as in \eqref{def_h}. The second part of the statement follows from \cref{trivial_integral}.
\end{proof} 

\section{Tame geometry and o-minimal structures}\label{sec:tame-sets}

We briefly recall the notions from o-minimal geometry that will be used in the sequel. For general background we refer to \cite{MR1633348}.

We begin with the notion of \emph{o-minimal structure}, starting from the class of semialgebraic sets. Recall that a subset of $\R^n$ is called \emph{semialgebraic} if it is obtained from finitely many polynomial equalities and inequalities by taking finite unions, finite intersections, and complements.

\begin{definition} \label{def:o-minimal_structure}
An \emph{o-minimal structure} (on the ordered field of real numbers) is a family
\[
\mathcal S=(\mathcal S_n)_{n\in\N},
\]
where each $\mathcal S_n$ is a collection of subsets of $\R^n$, satisfying the following properties:
\begin{enumerate}[(i)]
    \item every semialgebraic subset of $\R^n$ belongs to $\mathcal S_n$;
    \item if $A,B\in\mathcal S_n$, then $A\cup B$, $A\cap B$, and $\R^n\setminus A$ belong to $\mathcal S_n$;
    \item if $A\in\mathcal S_m$ and $B\in\mathcal S_n$, then $A\times B\in\mathcal S_{m+n}$;
    \item if $A\in\mathcal S_{n+1}$, then its image under a coordinate projection $\R^{n+1}\to\R^n$ belongs to $\mathcal S_n$;
    \item the sets in $\mathcal S_1$ are exactly the finite unions of points and open intervals.
\end{enumerate}
A subset of $\R^n$ is called \emph{$\mathcal{S}$-definable} if it belongs to $\mathcal S_n$. A map between definable sets is $\mathcal{S}$-\emph{definable} if its graph is $\mathcal{S}$-definable.
\end{definition}

Let $m\in\N$, $U\subseteq\R^m$ be an open neighborhood of $[-1,1]^m$, and $f_0 \colon U\to\R$ be a real-analytic function. The associated \emph{restricted analytic function} is the map $f:\R^m\to\R$ defined by
\[
f(x)=
\begin{cases}
f_0(x), & x\in[-1,1]^m,\\
0, & x\notin[-1,1]^m.
\end{cases}
\]

\begin{definition}[$\Ran$]
The o-minimal structure $\Ran$ is the smallest o-minimal structure containing, in addition to all semialgebraic sets, the graphs of all restricted analytic functions. 
\end{definition}

The $\Ran$-definable sets (resp.\ functions) are called \emph{globally subanalytic} sets (resp.\ functions); see also \cite[Section~1.1]{MR2769219}. Every $\Ran$-definable subset of $\R^n$ can be written as a finite union of sets of the form
\[
\pi\left(
\bigcap_{i=1}^r \Set{(x,y)\in\R^{n+\ell}: t_i(x,y)=0}
\cap
\bigcap_{j=1}^s \Set{(x,y)\in\R^{n+\ell}: u_j(x,y)>0}
\right),
\]
where $\pi:\R^{n+\ell}\to\R^n$ is a coordinate projection, and the $t_i,u_j$ are either polynomials or restricted analytic functions.

\begin{remark} \label{rem:ran}
    We stress the distinction between subanalytic and \emph{globally} subanalytic sets. Indeed, we require globally subanalytic sets to be defined by semialgebraic functions and \emph{restricted} analytic functions. This constraint turns the class defined in this way into an o--minimal family. Note that the zero-level set of the unrestricted sine function is infinite, which would violate axiom (v) of \cref{def:o-minimal_structure}.
\end{remark}

\begin{definition}[$\Ranexp$]
    The o-minimal structure $\Ranexp$ is the smallest o-minimal structure obtained from $\Ran$ by adjoining the (unrestricted) real exponential function. Equivalently, the $\Ranexp$-definable subsets of $\R^n$ are the sets obtained from semialgebraic sets, the graphs of restricted analytic functions, and the  whole graph $\{(x,e^x) : x\in \R\}$ of the real exponential function   by taking finite unions, finite intersections, complements, Cartesian products, and coordinate projections.
\end{definition}

We now recall the basic tame-geometric properties that will be used later on and that follows from o-minimality. 
First, we recall following basic regularity result for continuous definable maps. A simple proof is found in \cite[Lemma 7.6]{MR2823973}.
\begin{lemma}\label[lemma]{lem:ac}
Let $\mathcal{S}$ be an o-minimal structure. Let $I \subset \R$ be a compact interval. An $\mathcal{S}$-definable continuous function $f : I \to \R$ is absolutely continuous.
\end{lemma}

We recall the standard triangulation theorem for definable sets, see \cite[Ch.~8, (1.7)]{MR1633348} and \cite[Ch.~7, (3.2)]{MR1633348}. In what follows, a \emph{complex} in $\R^n$ is a subset of the simplices of a simplicial complex in $\R^n$. In particular, complexes in $\R^n$ are closed if and only if they are simplicial, thus containing all the faces of their simplices.

\begin{proposition} \label[proposition]{def_structure}
    Let $\mathcal{S}$ be an o-minimal structure. For each $\mathcal{S}$-definable set $A \subseteq \R^n$ there is a finite complex $K$ in $\R^n$ and a $\mathcal{S}$-definable homeomorphism $\varphi$ between the polyhedron $\lvert K \rvert$ and $A$ such that $\varphi$ is $C^1$ in every open simplex of $K$. In particular, each definable set is the finite union of some $C^1$ submanifolds. A \emph{(smooth) cell} $C \subseteq A$ is the image of an  open simplex of $K$ via such map. Note that cells are $\mathcal{S}$-definable.
\end{proposition}

In particular, a definable set $A \subseteq \R^n$ has a well-defined dimension $\dim(A)$, which coincides with the minimum $k$ such that $A$ is the finite union of $C^1$ submanifolds of dimension at most $k$, equivalently, such that $A$ is $k$-rectifiable.

We also recall the standard dimension inequality for definable maps, see \cite[Ch.~4, (1.6)]{MR1633348}.

\begin{proposition} \label[proposition]{def_image_dimension}
    Let $\mathcal{S}$ be an o-minimal structure and let $f \colon A \to B$ be a $\mathcal{S}$-definable map between $\mathcal{S}$-definable sets $A,B$. Then $f(A)$ is definable and  $\dim(f(A)) \le \dim(A)$.
\end{proposition}

We include a planar statement that is used in the proof of \cref{theo1_intro}.
\begin{lemma} \label{def_graph}
    Let $f \colon \R^2 \to \R$ be a non-zero analytic function and $r \ge 0$. Then
    \[
    Z_r \coloneq q  \Set{x \in \R^2 \colon \lvert x \rvert \le r,\; f(x)=0}
    \]
    is $\Ran$-definable of dimension at most one. Moreover, it admits a $\Ran$-definable geometric realization as finite graph that is absolutely continuous on each $1$-simplex.
\end{lemma}

\begin{proof}
    Clearly, the set $Z_r$ is $\Ran$-definable. Moreover, $Z_r$ has empty interior: otherwise $f$ would vanish on a non-empty open subset of $\R^2$, hence on all of $\R^2$ by analyticity, contradicting the assumption that $f \not\equiv 0$. Therefore $\dim(Z_r)\le 1$.

   Consequently, \cref{def_structure} ensures that $Z_r$ is $\Ran$-definably homeomorphic to (the polyhedron of) a complex $K$ of dimension at most one. Since $Z_r$ is compact, then $K$ is closed and therefore it is a finite simplicial complex. In particular, by Lemma \ref{lem:ac}, the homeomorphism is absolutely continuous on every simplex. Thus, $Z_r$ admits a $\Ran$-definable geometric realization $\varphi \colon |K| \to Z_r$ as finite graph, see \cref{geo_real}.
\end{proof}

The first technical tame-geometric tool that we use in the present work is the following Hardt's definable triviality theorem, see \cite[Ch.~9, (1.2)]{MR1633348}.
\begin{theorem} \label[theorem]{def_triviality}
    Let $\mathcal{S}$ be an o-minimal structure, $f \colon A \to B$ be a continuous $\mathcal{S}$-definable map between $\mathcal{S}$-definable sets $A,B$, and $C_1,\dots,C_N \subseteq A$ be $\mathcal{S}$-definable subsets. Then there exists a finite partition of $B$ into $\mathcal{S}$-definable sets $\{B_i\}_{i\in I}$, $\mathcal{S}$-definable sets $\{F_i\}_{i\in I}$, $\mathcal{S}$-definable subsets $F_{i,j}\subseteq F_i$, and $\mathcal{S}$-definable homeomorphisms
    \[
    \varphi_i \colon B_i \times F_i \to f^{-1}(B_i)
    \]
such that the following diagram commutes
\begin{equation*}
    \begin{tikzcd}
B_i\times F_i \arrow[rr, "\varphi_i"] \arrow[rd, "p_1"'] &     & f^{-1}(B_i) \arrow[ld, "f"] \\
                                                      & B_i &                            
\end{tikzcd}
\end{equation*}
and such that
    \[
    \varphi_i(B_i\times F_{i,j})=C_j\cap f^{-1}(B_i)
    \qquad\text{for every } i\in I,\ j\in\{1,\dots,N\}.
    \]
\end{theorem}

In our proof of \cref{theo1_intro}, we use \cref{def_triviality} to obtain uniform families of one-dimensional definable sets. In this regard, we also need precise tame-geometric properties for the multiplicative integral in the metabelian setting. We will summarize these properties in \cref{def_integral} and we discuss all the tame-geometric inputs needed for its proof in the following subsections.

\subsection{Tameness properties of the endpoint map} \label{sec:tamenessendpoint}

Our proof of \cref{def_integral} relies on the stability under integration result obtained in \cite{MR2769219}.
Following \cite[Definition~1.1 and Definition~1.2]{MR2769219}, if $S$ is a $\Ran$-definable set and $\ell\ge 0$, then for every Lebesgue measurable function $f\colon S\times\R^\ell \to\R$ one defines
\[
I_S(f):S\to \R,\quad I_S(f)(x)\coloneq 
\begin{cases}
\displaystyle\int_{\R^m} f(x,y)\,\de y, & \text{if } f(z,\cdot)\text{ is integrable for every $z\in S$},\\[0.8em]
0, & \text{otherwise.}
\end{cases}
\]
Moreover, one denotes by $\mathscr{C}(S)$ the $\R$-algebra generated by the $\Ran$-definable functions on $S$ and by the functions of the form $x\mapsto \log g(x)$, where $g\colon S\to(0,+\infty)$ is $\Ran$-definable. The elements of $\mathscr{C}(S)$ are called \emph{constructible functions}. Clearly, constructible functions are $\Ranexp$-definable.

\begin{theorem}[Cluckers--Miller, {\cite[Theorem~1.3]{MR2769219}}]\label[theorem]{CM_integration}
   Let $S$ be a $\Ran$-definable set and let $f\in \mathscr{C}(S\times\R^\ell)$. Then $
    I_S(f)\in \mathscr{C}(S).$
\end{theorem}

We only use \cref{CM_integration} in the following simple form.
\begin{corollary}\label[corollary]{def_primitive}
    Let $S$ be a $\Ran$-definable set and let $h \colon S \times [0,1] \to \R^m$ be a $\Ran$-definable map. Assume that $h(x,\cdot)$ is integrable for all $x\in S$. Then
    \[
    H \colon S \times [0,1] \to \R^m, \qquad
    (x,t) \mapsto \int_0^t h(x,\tau)\,\de \tau,
    \]
    is $\Ranexp$-definable.
\end{corollary}

\begin{proof}
    We argue component-wise, so we may assume $m=1$. Consider the $\Ran$-definable set
    \[
    E \coloneq  \Set{(x,t,\tau)\in S \times[0,1]\times[0,1] : 0\le \tau \le t},
    \]
    and the function 
    \[
    f\colon S\times[0,1]\times\R\to\R, \quad (x,t,\tau) \mapsto 
    \begin{cases}
    h(x,\tau), & \text{if } (x,t,\tau)\in E,\\
    0, & \text{otherwise.}
    \end{cases}
    \]
    Then $f$ is $\Ran$-definable, and therefore $f\in \mathscr{C}(S\times[0,1]\times\R)$. Moreover, by assumption for every $(x,t)\in S\times[0,1]$, the function $\tau\mapsto f(x,t,\tau)$ is integrable and
    \[
    I_{S\times[0,1]}(f)(x,t)=\int_0^t h(x,\tau)\,\de \tau.
    \]
    By \cref{CM_integration}, the latter function is in $\mathscr{C}(S\times[0,1])$. Since every constructible function is $\Ranexp$-definable, the claim follows.
\end{proof}

\begin{proposition} \label[proposition]{def_integral}
    Let $G$ be a simply connected, metabelian Lie group and consider mixed coordinates $\Phi:\g'\times B\to G$ as in \eqref{eq:mixed_coordiantes}. Let $\Lambda 
    $ be an $\Ran$-definable set and consider an $\Ran$-definable map
    \[
        \sigma \colon \Lambda \times [0,1] \to B,
    \]
    such that $\sigma(\lambda,\cdot) \colon [0,1] \to B$ is an absolutely continuous curve for every $\lambda \in \Lambda$, and such that $\sigma(\Lambda\times[0,1])$ is bounded. For every $\lambda\in \Lambda$, let $\gamma^\lambda$ be the multiplicative integral of $\sigma(\lambda,\cdot)$, see \cref{def:multiplicative_integral}, and define
    \[
        \gamma \colon \Lambda\times[0,1]\to G,
        \qquad
        \gamma(\lambda,t)\coloneq q \gamma^\lambda(t).
    \]
    Then the map
    \[
        \Phi^{-1}\circ\gamma \colon \Lambda \times [0,1] \to \g' \times B,
        \qquad
        (\lambda,t)\mapsto \Phi^{-1}(\gamma^\lambda(t)),
    \]
    is $\Ranexp$-definable. Equivalently, the family $\{\gamma^\lambda\}_{\lambda\in \Lambda}$ is $\Ranexp$-definable when $G$ is expressed in mixed coordinates.
\end{proposition}

\begin{proof}
Consider the partial differentiability set
\begin{equation*}
    D \coloneq \Set{(\lambda,t)\in \Lambda\times [0,1] : \sigma(\lambda,\cdot)\text{ is differentiable at $t$}}\subseteq \Lambda\times [0,1].
\end{equation*}
The set $D$ is $\Ran$-definable and the partial derivative $\partial_t\sigma :D\to B$ is $\Ran$-definable (see e.g.\ \cite[Lemma 6.8]{Coste}). Note that if $f:I\to \R$ is a definable continuous function on an open interval $I\subset \R$, then it is differentiable out of a finite set of points, \cite[Cor.\ 6.5]{Coste}. Hence, for all $\lambda\in  \Lambda$ the fiber of $D\subset \Lambda\times[0,1]$ over $\Lambda$ is the interval $[0,1]$ minus a finite set of points (that may depend on $\lambda$). Thus, define
\begin{equation}
    u : \Lambda \times [0,1] \to B, \qquad u(\lambda,t)\coloneq 
        \begin{cases}
\partial_t\sigma(\lambda,t), & \text{if }(\lambda,t)\in D,\\
0, & \text{otherwise}.
    \end{cases}
\end{equation}
As a piece-wise defined $\Ran$-definable function on $\Ran$-definable pieces, $u$ is $\Ran$-definable. Hence \(u(\lambda,\cdot)\) coincides almost everywhere with the derivative of
the absolutely continuous curve \(\sigma(\lambda,\cdot)\).
\[
\sigma(\lambda,t)-\sigma(\lambda,0)
=
\int_0^t u(\lambda,\tau)\,\de \tau,
\qquad
(\lambda,t)\in \Lambda \times[0,1].
\]
    Let $(a(\lambda,t),b(\lambda,t))\coloneq (\Phi^{-1}\circ\gamma)(\lambda,t)$ denote the mixed coordinates of $\gamma^\lambda(t)$. We need to prove that such function is $\Ranexp$-definable on $\Lambda\times [0,1]$. By \eqref{eq:curve_coordinates}, we have
    \[
    \begin{cases}
    a(\lambda,t)=\displaystyle\int_0^t \rho(\ad_{b(\lambda,\tau)})\,u(\lambda,\tau)\,\de \tau,\\[0.8em]
    b(\lambda,t)=\displaystyle\int_0^t u(\lambda,\tau)\,\de \tau
    =\sigma(\lambda,t)-\sigma(\lambda,0),
    \end{cases}
    \]
    where $\rho$ is the entire function defined as in \eqref{eq:power_series}. In particular, $b$ is $\Ran$-definable on $\Lambda\times[0,1]$. Since $\sigma(\Lambda\times[0,1])$ is bounded, also $b(\Lambda\times[0,1])$ is bounded. Set
    \[
    K\coloneq \overline{b( \Lambda\times[0,1])}\subseteq B.
    \]
    Then $K$ is a closed and bounded $\Ran$-definable set. The map
    \[
    B\to \End(\g),\qquad X\mapsto \rho(\ad_X),
    \]
    is real-analytic, being the composition of the linear map $X \mapsto \ad_X$ with the entire function $\rho$.
    
    Since $K$ is compact, there exists an open box $Q\subseteq B$ such that
    \[
    K\subseteq Q.
    \]
Each component of the restriction of $X\mapsto \rho(\ad_X)$ to $Q$ is a restricted analytic function. Hence the restriction of $X\mapsto \rho(\ad_X)$ to $K$ is $\Ran$-definable. In particular, the map
    \[
    (\lambda,\tau)\mapsto \rho(\ad_{b(\lambda,\tau)})
    \]
    is $\Ran$-definable on $\Lambda\times[0,1]$. It follows that the integrand
    \[
    (\lambda,\tau)\mapsto \rho(\ad_{b(\lambda,\tau)})\,u(\lambda,\tau)
    \]
    is $\Ran$-definable on $\Lambda\times[0,1]$. By \cref{def_primitive}, the first component $a$ is therefore $\Ranexp$-definable. Since $b$ is already $\Ran$-definable, the map
    \[
    (\lambda,\tau)\mapsto (a(\lambda,\tau),b(\lambda,\tau))
    \]
    is $\Ranexp$-definable, concluding the proof.
\end{proof}

\section{Main results} \label{main}

We are ready to prove our main results \cref{theo1_intro} and \cref{thm:rank_two_intro}. We point out that it is not restrictive to consider simply connected polarized groups $(G,V)$. Indeed, in view of \cref{rem:quotient}, a curve $\gamma$ is abnormal for $(G,V)$ if and only if its lift in the universal cover $\pi \colon \widehat{G} \to G$ is abnormal. Therefore, we get that $\abn(G,V)=\pi (\abn(\widehat{G},V))$. Being $\pi$ a local diffeomorphism, it preserves $k$-rectifiable sets, for each $k \in \Z$.

\subsection{The one-codimensional case} \label{sec:one_codim} We start by discussing the easier case of metabelian Lie groups such that $\dim(G/G')=1$. In what follows, if $\g$ is a Lie algebra and $X_1,\dots,X_k \in \g$, we use the shorthand notation $[X_1,\dots,X_k]$ for $\ad_{X_1} \circ \dots \circ \ad_{X_{k-1}}(X_k)$. For a subset $E \subseteq \g$, we denote with $I(E)$ the smallest ideal of $\g$ containing $E$.

\begin{proposition} \label[proposition]{prop:ideal_codim1}
    Let $\g$ be a Lie algebra such that $\dim(\g/\g')=1$, and $V \subseteq \g$ be a vector subspace such that $\Lie(V)=\g$. Then $I(V \cap \g') = \g'$. 
\end{proposition}

\begin{proof}
    Since $\g'$ is an ideal of $\g$, we get $I(V \cap \g') \subseteq  \g'$. For the other inclusion, we point out that, since $V$ is bracket-generating, then $V + \g' = \g$. In particular, $V \cap \g'$ has codimension $1$ in $V$. Let $\{Y_1,\dots,Y_m\}$ be a basis of $V \cap \g'$ and extend it to a basis $\{Y_1,\dots,Y_m, Y_{m+1}\}$ of $V$. Since $V$ is bracket-generating, $\g$ is spanned by elements of the form
    \begin{equation} \label{eq:span_codim1}
        [Y_{i_1}, \dots, Y_{i_k}], \quad \text{with $k \ge 1$.}
    \end{equation}
    The only case for an element as in  \eqref{eq:span_codim1} to be not in $\g'$ is when $k=1$ and $i_1 = m+1$. Moreover, since $[Y_{i_1}, \dots, Y_{i_{k-1}}, Y_{i_k}]=- [Y_{i_1}, \dots, Y_{i_k}, Y_{i_{k-1}}]$, we can further assume $i_k \neq m+1$ when $k \ge 2$. Therefore, since $\g'$ is one-codimensional, it is spanned by elements of the form
    \begin{equation} \label{eq:span_derived_codim1}
        [Y_{i_1}, \dots, Y_{i_k}], \quad \text{with $k \ge 1$ and $i_k \neq m+1$.}
    \end{equation}
    Since each element as in \eqref{eq:span_derived_codim1} belongs to $I(V \cap \g')$, we conclude that $\g' \subseteq I(V \cap \g')$.    
    \end{proof}

\begin{proposition} \label[proposition]{poly_codim1}
    Let $G$ be a metabelian Lie group with $\dim(G/G')=1$. Let $V$ be a bracket-generating polarization. Then for every $\lambda \in V^\perp\setminus\{0\}$, there exists $Y \in V \cap \g'$  such that $P^\lambda_{Y} \not\equiv 0$.
\end{proposition}

\begin{proof}
    Since $V$ is bracket-generating, we have that $\g = V + \g'$. Let $\lambda \in V^\perp$ and assume that $P^\lambda_{Y} \equiv 0$, for every $Y \in V\cap \g'$. Then, combining \cref{rem:abnormal_ideal} with \cref{prop:ideal_codim1}, we get that $P^\lambda_{Y} \equiv 0$ for every $Y \in \g'$. In particular, since $P^\lambda_{Y}(\id) = \lambda(Y)$, we obtain that $\lambda \in (\g')^\perp$.  We thus deduce that $\lambda = 0$.
\end{proof}


\begin{theorem} \label{thm:rank_one}
    Let $(G,V)$ be a metabelian, simply connected, polarized Lie group such that $\dim(G/G')=1$. Then the abnormal set $\abn(G,V)$ is either empty or it is contained in a subgroup of $\exp(V \cap \g')$ of dimension at most $\dim(G)-3$. In particular, the abnormal set $\abn(G,V)$ is $(\dim(G)-3)$-rectifiable and $(G,V)$ satisfies the Sard property.
\end{theorem}

\begin{proof}
    Fix $B \subseteq V$ to be a ($1$-dimensional) vector subspace such that $B \oplus \g' =\g$ and consider the mixed coordinates $\g' \times B$ on $G$ as in \eqref{eq:mixed_coordiantes}.  Fix a $\lambda$-abnormal curve $\gamma_u(t) =(a(t),b(t))$ in mixed coordinates, for some $\lambda \in \mathbb{S}(V^\perp)$. From \cref{poly_codim1}, there exists $Y \in V \cap \g'$  such that $P^\lambda_{Y} \not\equiv 0$. From \cref{poly_invariance}, we have $P^\lambda_Y(a,b) = P^\lambda_Y(0,b)$ and we observe that $P^\lambda_Y(0,\cdot) \colon B \cong \R \to \R$ is a non-constant analytic function, therefore its level sets are totally disconnected. Since $\gamma_u$ is $\lambda$-abnormal, we must have $P^\lambda_{Y}(0,b(t))=0$, for every $t \in [0,1]$. Since $b(0)=0$, the previous discussion implies that $b(t)=0$, for every $t \in [0,1]$ and therefore $\gamma_u(t) \in G'$, for every $t \in [0,1]$.

    We recall that $G'$ is an abelian Lie subgroup. Therefore, since $\gamma_u$ is horizontal, we get that $\gamma_u \subseteq \exp(V \cap \g')$. Since $V + \g'= \g$ and $\dim(\g/\g')=1$ , we get that $\dim(\exp(V \cap \g'))=\dim(V)-1$. If $\dim(V)=\dim(\g)$ then the abnormal set is empty, since $V^\perp = \{0\}$. If $\dim(V) \le \dim(\g) - 2$, then $\dim(\exp(V \cap \g') = \dim(V)-1 \le \dim(G) - 3$. We assume now  that $\dim(V) = \dim(\g) - 1$. In this case, if $\gamma_u$ is $\lambda$-abnormal, necessarily $\ker(\lambda)=V$. 
  By \eqref{eq:abn-equivalence}, and being $\gamma_u \subseteq G'$, we get  that $\gamma_u \subseteq \exp(W)$ where
\begin{equation*}
    W \coloneq  \Set{w \in V \cap \g' \colon \text{$[X,w] \in V$ for every $X \in V$}}.
\end{equation*}
Since $V$ is bracket generating, we infer that $W$ is a proper subspace of $V \cap \g'$, and therefore, being $\dim(\g'\cap V)= \dim(G)-2$, we get $\dim(W) \le \dim(G)-3$.
\end{proof}

\subsection{Proof of \cref{theo1_intro}}

    Let $(G,V)$ be a simply connected, polarized Lie group such that $G$ is metabelian and $\dim(V)=2$. Since $V$ is bracket-generating, then $V + \g'=\g$ and therefore $\dim(G/G') \le 2$. If $\dim(G/G') = 0$, then $G$ cannot be solvable. The case $\dim(G/G') = 1$ has been addressed in 
    \cref{sec:one_codim}. 
    We then consider the case $\dim(G/G') = 2$, which implies $V \oplus \g'=\g$. We consider mixed coordinates $\g' \times B$ on $G$ as in \eqref{eq:mixed_coordiantes}, with $B=V$. 

    We fix a basis $V = \spn\{X_1,X_2\}$, and we define 
    \[
    Y \coloneq  [X_1,X_2], \quad \Lambda \coloneq  \mathbb{S}(\Set{X_1,X_2,[X_1,X_2]}^\perp) \subseteq \mathbb{S} (\g^*).
    \]
    We point out that if $V=\g$, then the abnormal set is empty. Otherwise, since $V$ is bracket-generating, we get $Y \notin V$. The only case where $\g = \spn\{X_1,X_2,Y\}$ and $\dim(\g/\g')=2$ is the sub-Riemannian Heisenberg group, for which it is well known that the abnormal set consists of the identity element only. In every other case, $\Lambda$ is non-empty and $\dim(\Lambda)=\dim(G)-4$. By \cref{prop:goh_2}, since we are in the rank $2$ case, the abnormal set of $(G,V)$ is given by
    \begin{align} \label{eq:abn_2}
        \abn(G,V)  & =  \bigcup_{\lambda \in \Lambda} \Set{\End(u) \colon \text{$P^\lambda_{X_i}(\gamma_u(t))=P^\lambda_{Y}(\gamma_u(t))=0$, for every $t \in [0,1]$, $i\in\{1,2\}$}} \nonumber \\
        &=\bigcup_{\lambda \in \Lambda} \Set{\End(u) \colon \text{$P^\lambda_Y(\gamma_u(t))=0$, for every $t \in [0,1]$}},
    \end{align}
where in the second equality we used that $P^\lambda_{X_i}(\gamma_u(0))=0$, since $\lambda\in \Lambda$, together with the fact that, in view of \eqref{eq:diff}, the time derivative of $P^\lambda_{X_i}(\gamma_u(t))$ is a multiple of $P^\lambda_Y(\gamma_u(t))$. Note also that $P_Y^\lambda$ is not identically zero (otherwise, by \cref{rem:abnormal_ideal}, $\lambda$ would be zero). 

Since $Y \in \g'$ and $G$ is metabelian, from \cref{poly_invariance} we get that $P^\lambda_Y\colon \g' \times B \to \R$ only depends on the $B$ component. Therefore, we will effectively consider it as a function $P^\lambda_Y\colon B \cong \R^2 \to \R$.

    The outline of the proof is the following: first, we construct a countable family of $\Ranexp$-definable maps defined on $\Ranexp$-definable sets of dimension at most $\dim(G)-3$ parameterizing endpoints of abnormal geodesics. Being $\Ranexp$ an o-minimal structure and using \cref{def_image_dimension}, we conclude by proving that the union of images of these maps contains $\abn(G,V)$.

\medskip
    \paragraph{\textbf{Step 1: Constructing a countable family of tame maps}}
Fix a norm $\lvert \cdot \rvert$ on $B$. For $\lambda \in \Lambda$ and $r \in \N$, we define the $\Ran$-definable sets $Z^\lambda_r \coloneq  \{ b \in B \colon \lvert b \rvert \le r, \, P^\lambda_Y(b)=0\}$ and
\begin{equation*}
  E_r \coloneq  \Set{ (\lambda,b) \colon b \in Z^\lambda_r}\subseteq  \Lambda \times B ,
\end{equation*}
together with the projection maps $p \colon E_r \to  \Lambda$ and $q \colon E_r \to B$. Since $P^\lambda_Y$ is not identically zero, we recall from \cref{def_graph} that  $\dim(Z^\lambda_r)\leq 1$. Since $0\in Z_r^\lambda$, the subset
\[
E_r^0 \coloneq \Set{(\lambda,0) \colon \lambda\in \Lambda} \subseteq \Lambda \times B
\]
is $\Ran$-definable and contained in $E_r$. Applying \cref{def_triviality} to the map $p\colon E_r\to  \Lambda$, compatibly with the definable subset $E_r^0$ (meaning that $E_r^0$ plays the role of $C_1$ in the statement of the theorem), there exist a finite disjoint partition
\[
\Lambda=\bigsqcup_{i\in I_r} \Lambda_{i,r},
\]
for some set $I_r$, into $\Ran$-definable subsets, $\Ran$-definable sets $F_{i,r}\subseteq B$, and $\Ran$-definable homeomorphisms
\[
\varphi_{i,r}\colon \Lambda_{i,r}\times F_{i,r}\to p^{-1}(\Lambda_{i,r})\subseteq E_r
\]
such that $\varphi_{i,r}(\lambda,b)\in p^{-1}(\lambda)$ for every $i\in I_r$ and $(\lambda,b)\in \Lambda_{i,r}\times F_{i,r}$, and such that there exists $\xi_{i,r}\in F_{i,r}$ with
\begin{equation} \label{eq:zero_exists} 
\varphi_{i,r}(\lambda,\xi_{i,r})=(\lambda,0),
\quad\text{for every }\lambda\in \Lambda_{i,r}.
\end{equation}
We point out that $q\circ \varphi_{i,r}(\lambda,\cdot): F_{i,r} \to Z^\lambda_r$ is a $\Ran$-definable homeomorphism, 
for each $\lambda \in \Lambda_{i,r}$. In particular, $F_{i,r}$ is bounded and $\dim(F_{i,r})\le 1$. Therefore, by \cref{def_graph}, the set $F_{i,r}$ admits a $\Ran$-definable geometric realization as a finite graph. Choosing this geometric realization with $\xi_{i,r}$ as a vertex, we get a cell decomposition
\begin{equation*}
    F_{i,r} = \left( \bigsqcup_{n \in N_{i,r}} \alpha^n_{i,r} \right) \sqcup \left( \bigsqcup_{m \in M_{i,r}} \beta^m_{i,r}((0,1)) \right),
\end{equation*}
for some sets $M_{i,r}$, $N_{i,r}$, some $\alpha^n_{i,r} \in B$ and some $\Ran$-definable homeomorphisms
\[
\beta^m_{i,r} \colon [0,1] \to B.
\]
More precisely, each $\beta^m_{i,r}$ is the restriction to $[0,1]$ of the realization map on a $1$-simplex.

Moreover, there exists $n_0\in N_{i,r}$ such that $\alpha^{n_0}_{i,r}=\xi_{i,r}$. By \eqref{eq:zero_exists}, we get that
\[
\varphi_{i,r}(\lambda,\alpha^{n_0}_{i,r})=(\lambda,0)
\qquad\text{for every }\lambda\in \Lambda_{i,r}.
\]

    For every $i \in I_r$, $\lambda \in \Lambda_{i,r}$ and $\mu = (m,\varepsilon) \in M_{i,r} \times \{\pm1\}$ we define the absolutely continuous curve with values in $Z^\lambda_r$
    \begin{equation*}
        \sigma^{\lambda,\mu}_{i,r} \colon [0,1] \to B, \quad t \mapsto \begin{cases}
        q \circ \varphi_{i,r}(\lambda, \beta^m_{i,r}(t)), & \text{if $\varepsilon = +1$}, \\
        q \circ \varphi_{i,r}(\lambda, \beta^m_{i,r}(1-t)), & \text{if $\varepsilon = -1$},
        \end{cases}
    \end{equation*}
    and we denote by $\gamma^{\lambda,\mu}_{i,r}$ the multiplicative integral of $\sigma^{\lambda,\mu}_{i,r}$. We also set $g^{\lambda,\mu}_{i,r} \coloneq   \gamma^{\lambda,\mu}_{i,r}(1)$.
    
    We point out that, for every $\mu \in M_{i,r} \times \{\pm 1\}$, the map
    \begin{equation*}
        \sigma^\mu_{i,r} \colon \Lambda_{i,r} \times [0,1] \to B, \quad (\lambda,t) \mapsto \sigma^{\lambda,\mu}_{i,r}(t)
    \end{equation*}
    is a $\Ran$-definable  continuous map with bounded image. By \cref{lem:ac} for all $\lambda \in \Lambda$, $\sigma^{\mu}_{i,r}(\lambda,\cdot):[0,1]\to B$ is absolutely continuous. The assumptions of \cref{def_integral} are thus satisfied, and then the map
\begin{equation*}
    \Phi^{-1}\circ \gamma^m_{i,r} \colon \Lambda_{i,r} \times [0,1] \to \g' \times B, \quad (\lambda,t) \mapsto \Phi^{-1}( \gamma^\mu_{i,r}(t))
\end{equation*}
is $\Ranexp$-definable. In particular, evaluating at $t=1$, the map $\lambda \mapsto g^{\lambda,\mu}_{i,r}$ is $\Ranexp$-definable. Moreover, the group law is real-analytic in mixed coordinates, being obtained by conjugating the analytic multiplication map on $G$ by the analytic diffeomorphism $\Phi$. 
Therefore, for every $r,\ell \in \N$, $i \in I_r$, and $j \colon \{1,\dots,\ell\} \to M_{i,r} \times \{\pm 1\}$, the map
\begin{equation*}
    T_{i,r,\ell,j} \colon \Lambda_{i,r} \times [0,1] \to G,
    \quad
    (\lambda,t) \mapsto g^{\lambda,j(1)}_{i,r}\cdots g^{\lambda,j(\ell-1)}_{i,r}\gamma^{\lambda,j(\ell)}_{i,r}(t)
\end{equation*}
is $\Ranexp$-definable in mixed coordinates.
The first step of the proof is concluded by observing that $\dim(\Lambda_{i,r} \times [0,1]) \le \dim(\Lambda) + 1 = \dim(G)-3$.    
\medskip
\paragraph{\textbf{Step 2: Proving that the family of maps covers the abnormal set}}
    
    We fix $g \in \abn(G,V)$ and we claim that $g$ is in the image of some map $T_{i,r,\ell,j}$ constructed in Step 1. Since $g$ is in the abnormal set, there exists $\lambda \in \Lambda$ such that $g = \gamma_u(1)$, for some $\lambda$-abnormal curve $\gamma_u$, with $u \in L^1([0,1],V)$.
    We write $\gamma_u(t) =(a(t),b(t))$ in exponential coordinates as in \eqref{eq:curve_coordinates}, for some functions $a,b$.
    We observe that, since $B=V$ and in view of \eqref{eq:curve_coordinates}, the curve $\gamma_u$ is the multiplicative integral of $b$. 

Let $r \in \N$ and $i \in I_r$ such that $\lambda \in \Lambda_{i,r}$ and $b(t) \in Z^\lambda_r$ for all $t\in [0,1]$.

    Since $b(1)$ belongs to the connected component of $Z^\lambda_r$
 containing $0 \in B$, there exists a compatible sequence of curves $\sigma^{\lambda,\mu_1}_{i,r},\dots, \sigma^{\lambda,\mu_k}_{i,r}$, and $t \in [0,1]$, such that 
 \begin{equation*}
 \begin{cases}
     \sigma^{\lambda,\mu_1}_{i,r}(0)=0, \\ 
     \sigma^{\lambda,\mu_k}_{i,r}(t)=b(1),
 \end{cases}
 \end{equation*}
 in particular the curve $\eta : [0,1]\to B$ defined by the concatenation 
 \[
\eta \coloneq  \sigma^{\lambda,\mu_1}_{i,r} * \cdots * \sigma^{\lambda,\mu_{k-1}}_{i,r} * (\sigma^{\lambda,\mu_k}_{i,r} \circ \delta_t),
 \]
 where $\sigma^{\lambda,\mu_k}_{i,r} \circ \delta_t \colon s \mapsto \sigma^{\lambda,\mu_k}_{i,r}(ts)$, satisfies $\eta(1)=b(1)$ and $\eta ([0,1]) \subseteq  Z^\lambda_r$.
 
 We recall from \cref{def_graph} that $Z^\lambda_r$ admits a geometric realization as finite graph and therefore the fundamental group of each connected component is a finitely-generated free group. In particular, every homotopy class of loops based at $0$ contains a representative given by a finite concatenation of compatible curves among the arcs $\sigma^{\lambda,\mu}_{i,r}$, with $\mu\in M_{i,r} \times \{\pm1\}$. Let
\begin{equation*}
\rho \coloneq q  \sigma^{\lambda,\mu_{k+1}}_{i,r} * \cdots * \sigma^{\lambda,\mu_{k+h}}_{i,r},  \quad \rho([0,1])\subseteq Z^\lambda_r
\end{equation*}
be a closed curve based at $0 \in B$ and homotopically equivalent to $b * \overline{\eta}$. 
Then $b$ is homotopically equivalent to $\rho*\eta$ (in $Z^\lambda_r$) and, from \cref{prop:homotopy}, the end-point of $\gamma_u$, which is the multiplicative integral of $b$, is the end-point of the multiplicative integral of $\rho*\eta$. In view of \cref{rem:path2}, that end-point is
 \begin{equation*}
     g^{\lambda,\mu_{k+1}}_{i,r} \cdots g^{\lambda,\mu_{k+h}}_{i,r} g^{\lambda,\mu_1}_{i,r} \cdots g^{\lambda,\mu_{k-1}}_{i,r}  \gamma^{\lambda,\mu_k}_{i,r}(t) = T_{i,r,\ell,j}(\lambda,t),
 \end{equation*}
 for $\ell = k+h$ and some $j \colon \{1,\dots,\ell\} \to M_{i,r} \times \{\pm1\}$. We then proved that $\gamma_u(1)$ is in the imagine of some map $T_{i,r,\ell,j}$. We thus proved that
\[
\abn(G,V)\subseteq \bigcup_{r,\ell\in\N}\ \bigcup_{i\in I_r}\ \bigcup_{j:\{1,\dots,\ell\}\to M_{i,r} \times \{\pm1\}}
T_{i,r,\ell,j}(\Lambda_{i,r}\times[0,1]).
\]

By Step~1, for every $r,\ell\in\N$, $i\in I_r$, and $j:\{1,\dots,\ell\}\to M_{i,r} \times \{\pm1\}$, the set $\Lambda_{i,r}\times[0,1]$ is $\Ranexp$-definable and $\dim(\Lambda_{i,r}\times[0,1])\le \dim(G)-3$. Since $\Phi^{-1}\circ T_{i,r,\ell,j}$ is $\Ranexp$-definable in mixed coordinates, \cref{def_image_dimension} yields
\[
\dim\bigl(T_{i,r,\ell,j}( \Lambda_{i,r}\times[0,1])\bigr)\le \dim(G)-3.
\]

Therefore, each image $T_{i,r,\ell,j}(\Lambda_{i,r}\times[0,1])$ is the finite union of $C^1$ submanifolds of dimension at most $(\dim(G)-3)$. Finally, the above family is countable, hence $\abn(G,V)$ is $(\dim(G)-3)$-rectifiable. This concludes the proof. \hfill \qedsymbol

\subsection{Proof of \cref{thm:rank_two_intro}}

Set $I \coloneq  I(V \cap \g')$ and consider the normal integral subgroup $N \le G$ generated by $\exp(I)$. From \cref{rem:integral}, we get that $N$ is a Lie subgroup.

    A Goh-abnormal curve $\gamma_u$ is in particular $\lambda$-abnormal for some $\lambda \in \mathbb{S}(V^\perp)$. Therefore, $\gamma$ is contained in at least one of the following sets:
    \begin{align*}
        A_1 &\coloneq  \Set{\End(u) \colon \text{$\gamma_u$ Goh-abnormal and $\lambda$-abnormal for some $\lambda \in \mathbb{S}(V^\perp) \cap \mathbb{S}( I^\perp)$}}, \\
        A_2 &\coloneq  \Set{\End(u) \colon \text{$\gamma_u$ Goh-abnormal and $\lambda$-abnormal for some $\lambda \in \mathbb{S}(V^\perp) \setminus \mathbb{S}(I^\perp)$}}.
    \end{align*}
    Therefore,  $\abn^{Goh}(G,V) \subseteq A_1 \cup A_2$ and it is sufficient to prove that both $A_1$ and $A_2$ are $(\dim(G)-3)$-rectifiable.
    \medskip
\paragraph{\textbf{Case 1}}
Let $g \in A_1$. Then $g = \End(u)$ for some Goh-abnormal curve $\gamma_u$ that is $\lambda$-abnormal with $\lambda \in \mathbb{S}(V^\perp) \cap \mathbb{S}(I^\perp)$. In view of \cref{rem:quotient}, the curve $\gamma_u$ is abnormal when composed with the quotient map $\pi \colon G \to G/N$.
     In particular $\pi(g) \in \abn(G/N,V/I)$ and therefore $A_1 \subseteq \pi^{-1}\left(\abn(G/N,V/I)\right)$. We observe that $\dim(V/I)=\dim(V) - \dim(V \cap \g')=2$, since $V + \g' = \g$ and $\dim(\g/\g')=2$. 
    From \cref{theo1_intro}, we get that $\abn(G/N,V/I)$ is $(\dim(G/N) - 3)$-rectifiable. 
    We conclude that $A_1$ is $(\dim(G) - 3)$-rectifiable.

\medskip
\paragraph{\textbf{Case 2}} 
Let $g \in A_2$. Then $g = \End(u)$ for some Goh-abnormal curve $\gamma_u$ that is $\lambda$-abnormal with $\lambda \in \mathbb{S}(V^\perp) \setminus \mathbb{S}(I^\perp)$.
     Fix $B = \spn\{X_1,X_2\} \subseteq V$ to be a vector subspace such that $B \oplus \g' =\g$. Consider the mixed coordinates $\g' \times B$ on $G$ as in \eqref{eq:mixed_coordiantes} and write the curve $\gamma_u(t) =(a(t),b(t))$ as in \eqref{eq:curve_coordinates}. By \cref{rem:abnormal_ideal}, together with the fact that $\lambda \notin I(V \cap \g')^\perp$, there exists $Y \in V \cap \g'$ such that $P^\lambda_Y \not\equiv 0$. From \cref{poly_invariance}, we have $P^\lambda_Y(a,b) = P^\lambda_Y(0,b)$ and we observe that $P_Y^\lambda(0,\cdot) \colon B \cong \R^2 \to \R$ is a non-constant analytic function, therefore the singular points of its zero level set are totally disconnected. Since $\gamma_u$ is Goh-abnormal, then $P^\lambda_Y(0,b(t))=0$, moreover $0=P^\lambda_{[X_i,Y]}(0,b(t)) = X_i P^\lambda_Y(0,b(t))$ for $i=1,2$. Therefore, $b(t)$ is a singular point of the zero level set of $P_Y^\lambda(0,\cdot)$, for every $t \in [0,1]$. 
    
    From the previous discussion we get that $b(t)=0$, and therefore $\gamma_u \subseteq G'$. Moreover, we recall that $G'$ is an abelian Lie subgroup. Therefore, since $\gamma_u$ is horizontal, we get that $\gamma_u \subseteq \exp(V \cap \g')$. We thus proved that $A_2 \subseteq \exp(V \cap \g')$.
     
    Since $V + \g'= \g$ and $\dim(\g/\g')=2$ , we get that $\dim(V \cap \g')=\dim(V)-2$. If $\dim(V)=\dim(\g)$ then the abnormal set is empty, since $V^\perp = \{0\}$. If $\dim(V) \le \dim(\g) - 1$, then $\dim(V \cap \g') = \dim(V)-2 \le \dim(G) - 3$ and we get that $A_2$ is  $(\dim(G)-3)$-rectifiable. \hfill \qedsymbol

\subsection{Proof of \cref{prop:free-metabelian-sharpness}}

First, we define the polarized groups in the statement of \cref{prop:free-metabelian-sharpness}. For $s \ge 2$, we set $I_s \coloneq  \{(\alpha,\beta)\in\mathbb N^2:\alpha + \beta \leq s-2\}$.

\begin{definition}
     The free-metabelian Lie group $M_{2,s}$ of rank $2$ and step $s \ge 2$ is the simply connected Lie group whose Lie algebra $\mathfrak{m}_{2,s}$ admits a basis 
     $\{X_1,X_2\} \cup \{Y_{(\alpha,\beta)}\}_{(\alpha,\beta) \in I_s}$ for which the only non-trivial bracket relations between its elements are
     \begin{equation} \label{eq:relation_free}
         [X_1,X_2]=Y \coloneq  Y_{0,0}, \quad [X_1,Y_{\alpha,\beta}]=Y_{\alpha+1,\beta}, \quad [X_2,Y_{\alpha,\beta}] =Y_{\alpha,\beta+1}, \quad \text{if $\alpha+\beta < s-2$.}
     \end{equation}
     Moreover, $V \coloneq  \spn\set{X_1,X_2}$ is the first layer of a stratification of $\mathfrak{m_{2,s}}$. In particular, $(M_{2,s},V)$ is a stratified group.
\end{definition}

In the latter definition, the Lie algebra $\mathfrak{m}_{2,s}$ is indeed a metabelian Lie algebra (see \cite{MR159847,MR2905023}), it has nilpotency step $s$, and its dimension is \begin{equation*}
   \dim(\mathfrak{m}_{2,s}) = \dim(M_{2,s}) = \frac{s(s-1)}{2} +2.
  \end{equation*} 

Since the Lie group $M_{2,2}$ corresponds to the Heisenberg group, where the statement is well-known, we fix $s\geq 3$. We write abnormal polynomials in mixed coordinates as in \cref{sec:mixed}, by choosing $B$ as the first stratum of $\mathfrak{m}_{2,s}$, i.e., $B=\spn\{X_1,X_2\}$.
For $b=x_1X_1+x_2X_2$ and $\lambda\in\{X_1,X_2,Y\}^{\perp}$, we have
\[
 P_Y^\lambda(0,b)
 =\lambda\bigl(e^{\ad_b}Y\bigr)
 =\sum_{(\alpha,\beta)\in I_n}
   \frac{\lambda(Y_{\alpha,\beta})}{\alpha!\beta!}x_1^\alpha x_2^\beta.
\]
Thus, the functions $P_Y^\lambda(0,\cdot) \colon \R^2 \to \R$ are precisely the nonzero
polynomials of degree at most $s-2$ that vanish at the origin. By the characterization of the abnormal set in rank two established in \eqref{eq:abn_2}, the multiplicative integral of every planar curve contained in the zero set of one of these polynomials, for $\lambda \neq 0$, is abnormal.

Next, we need a preliminary lemma about the multiplicative integrals of planar curves written as a graph in free-metabelian Lie groups of rank $2$.

\begin{lemma} \label{multiplicative_graph}
    Let $f \colon [0,1] \to \R$ be absolutely continuous, with $f(0)=0$. Then, the mixed coordinates $(a(t),b(t))$ of the multiplicative integral $\gamma$ of $\sigma(t) \coloneq  tX_1 + f(t)X_2$ are given by
    \begin{equation*}
        a(t)= \sum_{(\alpha,\beta) \in I_s} \frac{Y_{\alpha,\beta}}{(\alpha+\beta+2)\alpha!\beta!}\int^t_0 \bigg(x^{\alpha+1}f(x)^\beta f'(x) - x^\alpha f(x)^{\beta+1}\bigg)\, \de x , \quad b(t) = \sigma(t).
    \end{equation*}
\end{lemma}

\begin{proof}
    Clearly, the control of $\gamma$ is $X_1+f'(x)X_2$ and  $b(t)=\sigma(t)$. We compute $a(t)$ using \eqref{eq:curve_coordinates}:
    \begin{align*}
        a(t) &= \int^t_0 \sum_{k=0}^{s-1} \frac{k}{(k+1)!} \ad^k_{xX_1+f(x)X_2}(X_1 + f'(x)X_2) \de x, \\
        &= \int^t_0 \sum_{k=1}^{s-1} \frac{k}{(k+1)!} (xf'(x)-f(x)) \ad^{k-1}_{xX_1+f(x)X_2}(Y) \de x, \\
        &= \int^t_0 \sum_{k=1}^{s-1} \sum_{j=0}^{k-1} \frac{k}{(k+1)!} \binom{k-1}{j} (xf'(x)-f(x))x^jf(x)^{k-1-j}Y_{j,k-1-j} \de x, \\
        &= \sum_{(\alpha,\beta) \in I_s} \frac{Y_{\alpha,\beta}}{(\alpha+\beta+2)\alpha!\beta!}\int^t_0 \bigg(x^{\alpha+1}f(x)^\beta f'(x) - x^\alpha f(x)^{\beta+1}\bigg)\, \de x,
    \end{align*}
    where we used the relations \eqref{eq:relation_free} and the change of indices $\alpha = j$, $\beta = k-1-j$.
\end{proof}

We set $\lambda_0\in\{X_1,X_2,Y\}^{\perp}$ to be the covector such that 
\begin{equation*}
    P^{\lambda_0}_Y(x_1,x_2) = (1-x_1^{s-3})x_2-x_1,
\end{equation*}
and $\Lambda \coloneq  \{X_1,X_2,Y,Y_{0,1}\}^{\perp}$, so that the $x_2$ coefficient of $P^\lambda_Y$ is zero, for every $\lambda \in \Lambda$. In what follows, we use that $P^{\lambda_0}_Y$ is irreducible, as explained in the following remark.
\begin{remark} \label{rem:irreducible}
    We claim that the polynomial $(1-x_1^{s-3})x_2-x_1$ is irreducible in $\R[x_1,x_2]$. Indeed, the claim is obvious if $s=3$. If $s >3$, since $(1-x_1^{s-3})x_2-x_1$ has degree one in $x_2$, it is irreducible in $\R(x_1)[x_2]$. Since $\gcd(1-x_1^{s-3},x_1)=1$, then it is also primitive in $\R[x_1][x_2]$. From Gauss's Lemma \cite[Exercise~3.4]{MR1322960}, we get that $(1-x_1^{s-3})x_2-x_1$ is irreducible in $\R[x_1,x_2]$.
\end{remark}

We define the smooth map
\begin{equation*}
    F \colon \R^2 \times \Lambda \to \R, \quad (x_1,x_2,\lambda) \mapsto P^{\lambda_0+\lambda}_Y(x_1,x_2),
\end{equation*}
and we observe that the function
\begin{equation*}
    f_0 \colon [0,\tfrac{1}{2}] \to \R, \quad x \mapsto \frac{x}{1 - x^{s-3}},
\end{equation*}
satisfies $F(x,f_0(x),0)=0$, for every $x \in [0,\tfrac{1}{2}]$. Moreover, $\partial_{x_2}F(x,f_0(x),0)=1-x^{s-3}\neq 0$, for every $x \in [0,\tfrac{1}{2}]$. A standard application of the implicit function theorem shows that there exists a neighborhood $U \subseteq \Lambda$ of zero such that, for every $\lambda \in U$, there exists a unique (smooth) function $f_\lambda \colon [0,\tfrac{1}{2}] \to \R$, with $f_\lambda(0)=0$, such that
\begin{equation} \label{eq:def_f_lambda}
    F(x,f_\lambda(x),\lambda)=0, \quad \text{for every $x \in [0,\tfrac{1}{2}]$}.
\end{equation}
Since $F$ is smooth, then $\lambda \mapsto f_\lambda(x)$ is also smooth, for every $x \in [0,\tfrac{1}{2}]$. 

For $\lambda \in U$, let $\gamma_\lambda \colon [0,\tfrac{1}{2}] \to M_{2,s}$ denote the multiplicative integral of $t \mapsto tX_1 + f_\lambda(t)X_2$. Since by construction $P^{\lambda_0+\lambda}(0,tX_1+f_\lambda(t)X_2)=0$, we get that $\gamma_\lambda$ is $(\lambda_0+\lambda)$-abnormal. Therefore the smooth map
\begin{equation*}
    \Psi \colon (0,\tfrac{1}{2}) \times U \to M_{2,s}, \quad (t,\lambda) \mapsto \gamma_\lambda(t),
\end{equation*}
is valued in $\abn(M_{2,s},V)$. Since $\dim((0,\tfrac{1}{2})\times \Lambda) = \dim(M_{2,s})-3$, it is sufficient to prove that $\de\Psi_{(t^*,0)}$ is injective,
for some $t^* \in (0,\tfrac{1}{2})$. We write $\Psi$ in mixed coordinates using \cref{multiplicative_graph}:
\begin{equation*}
    \Psi(t,\lambda)= \left(\sum_{(\alpha,\beta) \in I_s} c_{\alpha,\beta}Y_{\alpha,\beta}\int^t_0 \bigg(x^{\alpha+1}f_\lambda(x)^\beta f_\lambda'(x) - x^\alpha f_\lambda(x)^{\beta+1}\bigg)\, \de x, tX_1 + f_\lambda(t)X_2\right),
\end{equation*}
with $c_{\alpha,\beta} \neq 0$, for every $(\alpha,\beta) \in I_s$.

Next, we fix $t^* \in (0,\tfrac{1}{2})$ and $(\tau,\mu)$ to be in the kernel of $\de\Psi_{(t^*,0)}$. In particular, since $\mu \in \Lambda$, the $x_2$ coefficient of $P^\mu_Y$ is zero. We set
\begin{equation*}
    h(x) \coloneq  \left.\frac{\de}{\de\varepsilon}\right|_{\varepsilon=0}
 f_{\varepsilon \mu}(x).
\end{equation*}
Differentiating the $X_1$ component of $\Psi(t,\lambda)$ along $(\tau,\mu)$ at $(t^*,0)$, we get $\tau=0$. Doing the same for the $X_2$ component we get $h(t^*)=0$. From \eqref{eq:def_f_lambda}, we also get,
\begin{align}
    0 &= \left.\frac{\de}{\de\varepsilon}\right|_{\varepsilon=0} F(x,f_{\varepsilon\mu}(x),\varepsilon\mu) \notag \\ 
    &= \left.\frac{\de}{\de\varepsilon}\right|_{\varepsilon=0} \left[(1-x^{s-3})f_{\varepsilon\mu}(x) -x + P^{\varepsilon\mu}_Y(x,f_{\varepsilon\mu}(x))\right] \notag \\ 
    &= (1-x^{s-3})h(x) + P^\mu_Y(x,f_0(x)), \quad \text{for every $x \in (0,\tfrac{1}{2})$.}\label{eq:diff_F}
\end{align}
Differentiating the $Y_{\alpha,\beta}$ component of $\Psi(t,\lambda)$ we thus get
\begin{align}
    0 &= \left.\frac{\de}{\de\varepsilon}\right|_{\varepsilon=0} \int^{t^*}_0 \bigg(x^{\alpha+1}f_{\varepsilon\mu}(x)^\beta f_{\varepsilon\mu}'(x) - x^\alpha f_{\varepsilon\mu}(x)^{\beta+1}\bigg)\, \de x, \notag \\
    &= \left.\frac{\de}{\de\varepsilon}\right|_{\varepsilon=0} \bigg(\frac{1}{\beta+1} x^{\alpha+1}f_{\varepsilon\mu}(x)^{\beta +1}\bigg|^{t^*}_0 - \frac{\alpha+\beta+2}{\beta+1} \int^{t^*}_0 x^\alpha f_{\varepsilon\mu}(x)^{\beta+1}\, \de x \bigg), \notag \\
    &= -(\alpha+\beta+2)\int^{t^*}_0 x^\alpha f_0(x)^{\beta}h(x)\, \de x, \quad \text{for every $(\alpha,\beta) \in I_s$.} \label{eq:diff_Y}
\end{align}
where we first integrated the first term by parts and then used that $h(t^*)=0$.

Since $P^\mu_Y$ is a linear combination of monomials $x_1^\alpha x_2^\beta$, with $(\alpha,\beta) \in I_s$, from \eqref{eq:diff_Y} we get
\begin{equation*}
    0 = \int^{t^*}_0 P^\mu_Y(x,f_0(x))h(x) \, \de x \stackrel{\eqref{eq:diff_F}}{=} \int^{t^*}_0 (x^{s-3}-1)h(x)^2 \, \de x.
\end{equation*}
 If $s>3$, then $(x^{s-3}-1) < 0$ on $(0,t^*)$, we get that $h(x) \equiv 0$ on $[0,t^*]$, and from \eqref{eq:diff_F}, we infer that $P^\mu_Y(x,f_0(x)) \equiv 0$ on $ [0,t^*]$. If $s=3$, the latter directly follows from \eqref{eq:diff_F}. We thus proved that $P^\mu_Y=0$ in a subset of maximal dimension of the zero-level set of $P^{\lambda_0}_Y$, which is irreducible thanks to  \cref{rem:irreducible}. Therefore, from the real-Nullstellensatz \cite[Theorem~4.5.1]{MR1659509}, we get that $P^{\lambda_0}_Y$ divides $P^\mu_Y$.

Since $\deg(P^{\lambda_0}_Y)=s-2$ and $\deg(P^\mu_Y)\le s-2$, we get that $P^\mu_Y = c P^{\lambda_0}_Y$, for some $c \in \R$. We point out that the $x_2$ coefficient of $P^\mu_Y$ is zero while the $x_2$ coefficient of $P^{\lambda_0}_Y$ is not. We infer that $\mu=0$ and therefore $\de\Psi_{(t^*,0)}$ is injective. 

The map $\Psi$ is a smooth embedding from a neighborhood of $(t^*,0)$.
Its image is therefore a smooth submanifold inside
$\abn(M_{2,s},V)$ of dimension $\dim((0,\tfrac{1}{2})\times \Lambda) = \dim(M_{2,s})-3$. \hfill \qedsymbol

\bibliographystyle{alpha}	
\bibliography{biblio}
		
\end{document}